\documentclass[12pt,oneside,english]{amsart}
\usepackage{lmodern}

\usepackage[T1]{fontenc}
\usepackage[latin9]{inputenc}
\usepackage{geometry}
\usepackage{babel}
\usepackage{amstext}
\usepackage{amsthm}
\usepackage{amssymb}
\usepackage{setspace}
\usepackage{microtype}
\usepackage[unicode=true,pdfusetitle,
 bookmarks=true,bookmarksnumbered=false,bookmarksopen=false,
 breaklinks=false,pdfborder={0 0 1},backref=false,colorlinks=false]
 {hyperref}

\usepackage[nameinlink,capitalize]{cleveref}

\makeatletter

\numberwithin{equation}{section}
\numberwithin{figure}{section}
\theoremstyle{plain}
\newtheorem{thm}{\protect\theoremname}[section]
\theoremstyle{plain}
\newtheorem{conjecture}[thm]{\protect\conjecturename}
\theoremstyle{definition}
\newtheorem{defn}[thm]{\protect\definitionname}
\theoremstyle{plain}
\newtheorem{fact}[thm]{\protect\factname}
\theoremstyle{remark}
\newtheorem{rem}[thm]{\protect\remarkname}
\theoremstyle{plain}
\newtheorem{cor}[thm]{\protect\corollaryname}
\theoremstyle{plain}
\newtheorem{prop}[thm]{\protect\propositionname}
\theoremstyle{plain}
\newtheorem{lem}[thm]{\protect\lemmaname}
\theoremstyle{remark}
\newtheorem{notation}[thm]{\protect\notationname}
\theoremstyle{remark}
\newtheorem{claim}[thm]{\protect\claimname}
\theoremstyle{plain}
\newtheorem{obs}[thm]{Observation}

\@ifundefined{date}{}{\date{}}
\usepackage{amsmath}

\usepackage{autobreak}

\usepackage{xcolor}
\definecolor{blue}{RGB}{14,107,217}
\definecolor{green}{RGB}{0,158,40}
\definecolor{red}{RGB}{235,16,16}
\definecolor{brown}{RGB}{164,66,0}
\definecolor{orange}{RGB}{231,135,26}
\definecolor{purple}{RGB}{94,53,177}

\usepackage{amssymb}

\makeatother

\providecommand{\claimname}{Claim}
\providecommand{\conjecturename}{Conjecture}
\providecommand{\corollaryname}{Corollary}
\providecommand{\definitionname}{Definition}
\providecommand{\factname}{Fact}
\providecommand{\lemmaname}{Lemma}
\providecommand{\notationname}{Notation}
\providecommand{\propositionname}{Proposition}
\providecommand{\remarkname}{Remark}
\providecommand{\theoremname}{Theorem}

\begin{document}
\global\long\def\N{\mathbb{N}}%
\global\long\def\Z{\mathbb{Z}}%
\global\long\def\Q{\mathbb{Q}}%
\global\long\def\R{\mathbb{R}}%
\global\long\def\C{\mathbb{C}}%
 
\global\long\def\T{\mathbb{T}}%
\global\long\def\cM{\mathcal{M}}%
\global\long\def\cN{\mathcal{N}}%
\global\long\def\cC{\mathcal{C}}%
\global\long\def\cZ{\mathcal{Z}}%
\global\long\def\cQ{\mathcal{Q}}%
 
\global\long\def\cU{\mathcal{U}}%
\global\long\def\cE{\mathcal{E}}%
\global\long\def\cS{\mathcal{S}}%
\global\long\def\fC{\mathfrak{C}}%
\global\long\def\i#1{\iota\left(#1\right)}%

\title{On dp-minimal expansions of the integers II}
\author{Eran Alouf}
\address{Einstein Institute of Mathematics, Hebrew University of Jerusalem,
91904, Jerusalem Israel.}
\email{Eran.Alouf@mail.huji.ac.il}
\begin{abstract}
We first prove that if $\mathcal{Z}$ is a dp-minimal expansion of
$\left(\mathbb{Z},+,0,1\right)$ which is not interdefinable with
$\left(\mathbb{Z},+,0,1,<\right)$, then every infinite subset of
$\Z$ definable in $\mathcal{Z}$ is generic in $\Z$. Using this,
we prove that if $\mathcal{Z}$ is a dp-minimal expansion of $\left(\mathbb{Z},+,0,1\right)$
with monster model $G$ such that $G^{00}\neq G^{0}$, then for some
$\alpha\in\R\backslash\Q$, the cyclic order on $\Z$ induced by the
embedding $n\mapsto n\alpha+\Z$ of $\Z$ in $\R\big/\Z$ is definable
in $\cZ$. The proof employs the Gleason-Yamabe theorem for abelian
groups.
\end{abstract}

\maketitle

\section{Introduction}

In this paper we make another step towards classifying dp-minimal
expansions of $\left(\mathbb{Z},+,0,1\right)$. We refer to the introduction
of \cite{Alo2020} for a broader exposition. For two structures $\cM_{1}$,$\cM_{2}$
with the same underlying universe $M$, we say that $\cM_{1}$ is
a \emph{reduct} of $\cM_{2}$, and that $\cM_{2}$ is an \emph{expansion}
of $\cM_{1}$, if for every $k\in\N$, every subset of $M^{k}$ which
is definable in $\cM_{1}$ is also definable (with parameters) in
$\cM_{2}$. We say that $\cM_{1}$ and $\cM_{2}$ are \emph{interdefinable}
if $\cM_{1}$ is a reduct of $\cM_{2}$ and $\cM_{2}$ is a reduct
of $\cM_{1}$, and we say that $\cM_{2}$ is a \emph{proper expansion}
of $\cM_{1}$ if $\cM_{2}$ is an expansion of $\cM_{1}$ but $\cM_{1}$
and $\cM_{2}$ are not interdefinable. 

In \cite{ConantPillay2018} it was shown that $\left(\Z,+,0,1\right)$
has no stable proper expansions of finite dp-rank. Every known dp-minimal
proper expansion of $\left(\mathbb{Z},+,0,1\right)$ is a reduct of
a dp-minimal expansion of a structure from one of three specific families:

$\bullet$ The first such family consists of the single structure
$\left(\mathbb{Z},+,0,1,<\right)$. In \cite{Aschenbrenner_et_al_I_2015}
it was shown that $\left(\Z,+,0,1,<\right)$ has no dp-minimal proper
expansions. This was later significantly strengthened in \cite{DolichGoodrick2017}
(see \cref{no_strongly_dependent_expansions_of_the_order}). In
\cite{Conant2018_no_intermediate} it was shown that $\left(\Z,+,0,1,<\right)$
has no proper reducts which are proper expansions of $\left(\mathbb{Z},+,0,1\right)$.
Finally, \cite{Alo2020} classified $\left(\mathbb{Z},+,0,1,<\right)$
as the unique dp-minimal expansion of $\left(\mathbb{Z},+,0,1\right)$
which defines an infinite subset of $\mathbb{N}$, as well as the
unique dp-minimal expansion of $\left(\mathbb{Z},+,0,1\right)$ which
does not eliminate $\exists^{\infty}$.

$\bullet$ The second family consists of the structures $\left(\mathbb{Z},+,0,1,\preceq_{v}\right)$,
where $v$ is a generalized valuation. For every strictly descending
chain $\left(B_{i}\right)_{i<\omega}$ of subgroups of $\Z$ with
$B_{0}=\Z$, we can define a function $v:\Z\to\omega\cup\left\{ \infty\right\} $
by $v(x):=\max\left\{ i\in\omega\,:\,x\in B_{i}\right\} $. We call
such a function $v$ a generalized valuation on $\Z$. We also denote
by $\preceq_{v}$ the associated partial order, i.e., $a\preceq_{v}b$
if and only if $v\left(a\right)\leq v\left(b\right)$. When $p\in\N$
is prime and $B_{i}=p^{i}\Z$, $v$ is the usual $p$-adic valuation,
and we denote it by $v_{p}$ and the associated order by $\preceq_{p}$.
In \cite{AD19} it was shown that for every nonempty (possibly infinite)
set of primes $P\subseteq\mathbb{N}$, the structure $\left(\Z,+,0,1,\left\{ \preceq_{p}\right\} _{p\in P}\right)$
has dp-rank $\left\vert P\right\vert $. In particular, for a single
prime $p$, the structure $\left(\Z,+,0,1,\preceq_{p}\right)$ is
dp-minimal. This was generalized in \cite{Cla20}, where it was shown
that for every generalized valuation $v$, the structure $\left(\mathbb{Z},+,0,1,\preceq_{v}\right)$
is dp-minimal. In \cite{AD19} it was also shown that $\left(\Z,+,0,1,\preceq_{p}\right)$
has no proper reducts which are proper expansions of $\left(\mathbb{Z},+,0,1\right)$.

$\bullet$ The third family consists of the structures $\left(\mathbb{Z},+,0,1,C_{\alpha}\right)$,
where $\alpha\in\R\backslash\Q$ and $C_{\alpha}$ is the cyclic order
on $\mathbb{Z}$ induced by $\alpha$. Denote by $\cC$ the positively
oriented cyclic order on $\R\big/\Z$, i.e., for $u,v,w\in\left[0,1\right)$
we have $\cC\left(u+\Z,v+\Z,w+\Z\right)$ if and only if $u<v<w$
or $v<w<u$ or $w<u<v$. Let $\alpha\in\R\backslash\Q$ and let $\eta:\Z\rightarrow\R\big/\Z$
be given by $\eta\left(n\right):=n\alpha+\Z$. Then $C_{\alpha}$
is defined by setting $C_{\alpha}\left(a,b,c\right)$ if and only
if $\cC\left(\eta\left(a\right),\eta\left(b\right),\eta\left(c\right)\right)$.
In \cite{TW2023} it was shown that for every $\alpha\in\R\backslash\Q$,
the structure $\left(\mathbb{Z},+,0,1,C_{\alpha}\right)$ is dp-minimal.

These structures themselves are not the only dp-minimal expansions
of $\left(\mathbb{Z},+,0,1\right)$. As noted in \cite{TW2023}, the
Shelah expansion $\left(\mathbb{Z},+,0,1,C_{\alpha}\right)^{Sh}$
of $\left(\mathbb{Z},+,0,1,C_{\alpha}\right)$ is a proper expansion
of $\left(\mathbb{Z},+,0,1,C_{\alpha}\right)$. In \cite{Wal2020}
Walsberg constructed uncountably many proper dp-minimal expansions
of $\left(\mathbb{Z},+,0,1,C_{\alpha}\right)^{Sh}$. In the same paper,
Walsberg also constructed a proper dp-minimal expansion of $\left(\Z,+,0,1,\preceq_{p}\right)$.
In \cite{Cla20} it was shown that even if we expand $\left(\mathbb{Z},+,0,1,\preceq_{v}\right)$
by adding to the value sort all unary subsets and all monotone binary
relations, then the resulting structure is still dp-minimal. And,
of course, every reduct of any of these expansions is still dp-minimal.

These results showed that $\left(\mathbb{Z},+,0,1\right)$ has many
more dp-minimal expansions than was previously thought, and may give
the impression that the classification of dp-minimal expansions of
$\left(\mathbb{Z},+,0,1\right)$ is a lost cause. But, while a complete
classification might be beyond reach, we believe that the following
holds:
\begin{conjecture}
\label{main_classification_conjecture}Let $\cZ$ be a dp-minimal
proper expansion of $\left(\Z,+,0,1\right)$. Then exactly one of
the following holds:
\begin{itemize}
\item $\cZ$ is interdefinable with $\left(\Z,+,0,1,<\right)$.
\item There is $\alpha\in\R\backslash\Q$ such that $C_{\alpha}$ is definable
in $\cZ$.
\item There is a generalized valuation $v$ such that $\preceq_{v}$ is
definable in $\cZ$.
\end{itemize}
\end{conjecture}

We can try to prove \cref{main_classification_conjecture} by classifying
each case separately. \cite{Alo2020} already dealt with the first
case. In this paper, we deal with the second case by proving:
\begin{thm}
\label{main_theorem_cyclic_order}Let $\mathcal{Z}$ be a dp-minimal
expansion of $\left(\mathbb{Z},+,0,1\right)$, and let $G$ be a monster
model. Suppose that $G^{00}\neq G^{0}$. Then for some $\alpha\in\R\backslash\Q$,
$C_{\alpha}$ is definable in $\cZ$.
\end{thm}

We also show that the converse holds. In order to prove \cref{main_theorem_cyclic_order},
we first prove:
\begin{thm}
\label{main_theorem_genericity}Let $\mathcal{Z}$ be a dp-minimal
expansion of $\left(\mathbb{Z},+,0,1\right)$ which is not interdefinable
with $\left(\mathbb{Z},+,0,1,<\right)$. Then every infinite subset
of $\Z$ definable in $\mathcal{Z}$ is generic in $\Z$.
\end{thm}

This is a general theorem about dp-minimal expansion of $\left(\mathbb{Z},+,0,1\right)$,
and we expect it to also play a key role in dealing with the third
and last case of \cref{main_classification_conjecture}.

\section{Preliminaries}

\subsection{Dp-minimality}
\begin{defn}
Let $T$ be a theory, and let $\kappa$ be a cardinal. An \emph{ict-pattern
of depth $\kappa$ }consists of:

\begin{itemize}
\item a collection of formulas $\left(\phi_{\alpha}\left(x;y_{\alpha}\right)\,:\,\alpha<\kappa\right)$,
with $\left\vert x\right\vert =1$, and
\item an array $\left(b_{\alpha,i}\,:\,i<\omega,\,\alpha<\kappa\right)$
of tuples in some model $\cM$ of $T$, with $|b_{\alpha,i}|=|y_{\alpha}|$
\end{itemize}
such that for every $\eta:\kappa\to\omega$ there exists an element
$a_{\eta}\in\cM$ such that 
\[
\cM\models\phi_{\alpha}\left(a_{\eta};b_{\alpha,i}\right)\iff\eta\left(\alpha\right)=i
\]

We define $\kappa_{ict}=\kappa_{ict}\left(T\right)$ as the minimal
$\kappa$ such that there does not exist an ict-pattern of depth $\kappa$,
and define $\kappa_{ict}=\infty$ if there is no such $\kappa$. For
a structure $\cM$, we let $\kappa_{ict}\left(\cM\right):=\kappa_{ict}\left(Th\left(\cM\right)\right)$.
\end{defn}

\begin{defn}
A theory $T$ is \emph{dp-minimal} if $\kappa_{ict}\left(T\right)\le2$.
A structure $\cM$ is \emph{dp-minimal} if $Th\left(\cM\right)$ is.
\end{defn}

\begin{fact}[{\cite[Observation 4.13]{Simon_2015}}]
\label{NIP_equivalent_to_non_infinite_dp_rank}A theory $T$
is NIP if and only if $\kappa_{ict}\left(T\right)<\infty$.
\end{fact}

\begin{defn}
Let $T$ be a theory, and let $\kappa$ be a cardinal. An \emph{inp-pattern
of depth $\kappa$ }consists of:

\begin{itemize}
\item a collection of formulas $\left(\phi_{\alpha}\left(x;y_{\alpha}\right)\,:\,\alpha<\kappa\right)$,
with $\left\vert x\right\vert =1$, and
\item an array $\left(b_{\alpha,i}\,:\,i<\omega,\,\alpha<\kappa\right)$
of tuples in some model of $T$, with $|b_{\alpha,i}|=|y_{\alpha}|$
\end{itemize}
such that: 
\begin{enumerate}
\item for each $\alpha<\kappa$ there exists $k_{\alpha}\in\N$ for which
the row $\left\{ \phi_{\alpha}\left(x;b_{\alpha,i}\right)\,:\,i<\omega\right\} $
is $k_{\alpha}$-inconsistent, and
\item for every $\eta:\kappa\to\omega$ the path $\left\{ \phi_{\alpha}\left(x;b_{\alpha,\eta\left(\alpha\right)}\right)\,:\,\alpha<\kappa\right\} $
is consistent.
\end{enumerate}
We define $\kappa_{inp}=\kappa_{inp}\left(T\right)$ as the minimal
$\kappa$ such that there does not exist an inp-pattern of depth $\kappa$,
and define $\kappa_{inp}=\infty$ if there is no such $\kappa$.
\end{defn}

\begin{fact}[{\cite[Proposition 10]{Adl}}]
\label{k_inp_eq_k_ict_under_nip}For every theory $T$ we have
$\kappa_{inp}\le\kappa_{ict}$, and if $T$ is NIP then $\kappa_{inp}=\kappa_{ict}$.
\end{fact}

\subsection{Connected components and compact quotients}

Let $G$ be a $\emptyset$-definable group in a theory $T$.
\begin{defn}
Let $H$ be a subgroup of $G$ type-definable over a small set $A$.
If the set
\[
\left\{ \left[G\left(\cM\right):H\left(\cM\right)\right]\,:\,A\subseteq\cM\models T\right\} 
\]
is bounded, we say that \emph{$H$ has bounded index in $G$}. The
\emph{index} of $H$ in $G$ is the supremum of this set.
\end{defn}

\begin{rem}
If $H$ has bounded index in $G$, then the supremum is at most $2^{\left\vert T\right\vert }$,
and is obtained whenever $\cM$ is saturated enough. If in addition
$H$ is definable, then $\left[G\left(\cM\right):H\left(\cM\right)\right]$
is finite and does not depend on $\cM$.
\end{rem}

\begin{defn}
Let $A$ be a small set. Define: 
\begin{alignat*}{2}
 & G_{A}^{0} &  & :=\bigcap\left\{ H\,:\,H\text{ is an }A\text{-definable subgroup of finite index in }G\right\} \\
 & G_{A}^{00} &  & :=\bigcap\left\{ H\,:\,H\text{ is an }A\text{-type-definable subgroup of bounded index in }G\right\} 
\end{alignat*}
\end{defn}

Both $G_{A}^{0}$ and $G_{A}^{00}$ are $A$-type-definable and have
bounded index in $G$ (hence $G_{A}^{00}$ is the smallest such subgroup).
\begin{defn}
If $G_{A}^{00}$ (resp. $G_{A}^{0}$) does not depend on $A$, we
denote it by $G^{00}$ (resp. $G^{0}$) and say that $G^{00}$ (resp.
$G^{0}$) exists.
\end{defn}

If $G^{00}$ (resp. $G^{0}$) exists, then it is the intersection
of all type-definable (resp. definable) subgroups of $G$ of bounded
(resp. finite) index, and is a normal subgroup of $G$. If $G^{00}$
exists, then $G^{0}$ exists as well.
\begin{fact}[\cite{She2008}]
If $T$ is NIP then $G^{00}$ exists.
\end{fact}

Let $A$ be a small set, and $H$ an $A$-type-definable subgroup
of $G$ of bounded index. Then $G\left(\fC\right)\big/H\left(\fC\right)$
does not depend on the choice of the monster $\fC$, and we denote
it simply by $G\big/H$.

Let $\cM\supseteq A$ be a small model. We define a topology (called
the \emph{logic topology}) on $G\big/H$ by saying that $C\subseteq G\big/H$
is closed if and only if $\pi^{-1}\left(C\right)$ is $\cM$-type-definable,
where $\pi:G\rightarrow G\big/H$ is the quotient map.
\begin{fact}[{\cite[8.1.5]{Simon_2015}}]
This topology does not depend on $\cM$. In this topology, $G\big/G^{0}$
and $G\big/G^{00}$ are compact Hausdorff topological groups.
\end{fact}

\begin{rem}
\label{G0_in_any_expansion_of_Z}If $G$ is (a monser mode of)
any expansion of $\left(\Z,+\right)$, then $G^{0}$ exists and is
equal to $\bigcap_{m=1}^{\infty}mG$. This is because for each $m$,
$\Z$ has exactly one subgroup of index $m$, namely $m\Z$, and this
fact can be expressed by first-order formulas in $\left(\Z,+\right)$.
It follows that $G\big/G^{0}\cong\hat{\Z}=\underleftarrow{\lim}\Z\big/m\Z$.
\end{rem}

\subsection{Externally definable sets and the Shelah expansion}
\begin{defn}
Let $\cM$ be a structure in a language $L$, and fix an elementary
extension $\cN$ of $\cM$ which is $\left\vert \cM\right\vert ^{+}$-saturated. 
\begin{enumerate}
\item An \emph{externally definable} subset of $\cM$ is a subset of $\cM^{k}$
of the form 
\[
\phi\left(\cM;b\right):=\left\{ a\in\cM\,:\,\cN\models\phi\left(a;b\right)\right\} 
\]
for some $k\ge1$, $\phi\left(x;y\right)\in L$, and $b\in\cN^{\left\vert y\right\vert }$. 
\item The \emph{Shelah expansion} of $\cM$, denoted by $\cM^{Sh}$, is
defined to be the structure in the language $L^{Sh}:=\left\{ R_{\phi\left(x;b\right)}\left(x\right)\,:\,\phi\left(x;y\right)\in L,\,b\in\cN^{\left\vert y\right\vert }\right\} $
whose universe is $\cM$ and where each $R_{\phi\left(x;b\right)}\left(x\right)$
is interpreted as $\phi\left(\cM;b\right)$.
\end{enumerate}
Note that the property of being an externally definable subset of
$\cM$ does not depend on the choice of $\cN$. Hence, up to interdefinability,
$\cM^{Sh}$ does not depend on the choice of $\cN$ (although formally,
a different $\cN$ gives a different language and so a different structure),
so it makes sense to talk about \emph{the} Shelah expansion.
\end{defn}

\begin{fact}[{Shelah. See \cite[Proposition 3.23]{Simon_2015}}]
\label{nip_gives_shelah_expansion_QE}Let $\cM$ be NIP. Then
$\cM^{Sh}$ eliminates quantifiers.
\end{fact}

It follows that if $\cM$ is NIP then every definable set in $\cM^{Sh}$
is externally definable in $\cM$, and $\left(\cM^{Sh}\right)^{Sh}$
is interdefinable with $\cM^{Sh}$.
\begin{cor}
\label{shelah_expansion_preserves_dp_rank}Let $\cM$ be any structure.
Then $\kappa_{ict}\left(\cM^{Sh}\right)=\kappa_{ict}\left(\cM\right)$.
\end{cor}

\section{Genericity}

In this section we prove \cref{main_theorem_genericity}.
\begin{defn}
A subset $A$ of an abelian semigroup $\left(S,+\right)$ is called
\emph{generic in $S$} (or \emph{syndetic in $S$}, as this was called
in \cite[Subsection 4.1]{Alo2020}) if there is a finite subset $F\subseteq S$
such that $\bigcup_{s\in F}\left(A-s\right)=S$, where $A-s:=\left\{ a\in S:a+s\in A\right\} $.
\end{defn}

For $S=\N$, this is the same as saying that $A$ is infinite and
there is a uniform bound on the distance between every two consecutive
elements of $A$. For $S=\Z$, this is the same as saying that $\inf A=-\infty$,
$\sup A=\infty$, and there is a uniform bound on the distance between
every two consecutive elements of $A$.

Recall the following definitions from \cite[Subsection 4.2]{Alo2020}:
\begin{defn}
Let $A\subseteq\Z$. We say that $N\in\N$ is a \emph{bound on the
two-sided gaps of $A$} if for every $x\in A$ there exists $d\in\left[-N,-1\right]\cup\left[1,N\right]$
such that $x+d\in A$. We say that \emph{$A$ has bounded two-sided
gaps} if there exists a bound $N$ on the two-sided gaps of $A$.
\end{defn}

\begin{defn}
Let $A\subseteq\Z$.
\begin{enumerate}
\item Define a function $L_{A}:\Z\rightarrow\N\cup\{\infty\}$ by
\[
L_{A}(y):=\sup\left\{ m\in\N\,:\,\left[y-m,y-1\right]\cap A=\emptyset\right\} 
\]
\item Define a function $L_{A}:P(\Z)\rightarrow\N\cup\{\infty,-\infty\}$
by
\[
L_{A}(B):=\sup\left\{ L_{A}(y)\,:\,y\in B\text{ and }y>\inf A\right\} 
\]
\end{enumerate}
\end{defn}

\begin{rem}
\label{not_syndetic_implies_L_eq_infty}In \cite[Remark 4.5. (4)]{Alo2020}
it is remarked that if $A\subseteq\N$ is infinite and not generic
in $\N$ then $L_{A}\left(A\right)=\infty$. In fact, something a
bit more general is true: for $A\subseteq\Z$, if at lest one of the
sets $A\cap\N$, $\left(-A\right)\cap\N$ is infinite and not generic
in $\N$, then $L_{A}\left(A\right)=\infty$. 
\end{rem}

\begin{fact}[{\cite[Proposition 4.7., see remark below]{Alo2020}}]
\label{bounded_two_sided_gaps_gives_IP}Let $A\subseteq\Z$
be infinite with $L_{A}\left(A\right)=\infty$, and let $\mathcal{Z}:=\left(\mathbb{Z},+,0,1,A\right)$.
Suppose that every infinite subset of $A$ that is definable in $\mathcal{Z}$
has bounded two-sided gaps. Then the formula $y-x\in A$ has IP. 
\end{fact}

\begin{rem}
In \cite[Proposition 4.7.]{Alo2020}, instead of assuming that $A\subseteq\Z$
is infinite with $L_{A}\left(A\right)=\infty$, the assumption is
that $A\subseteq\N$ is infinite and not generic in $\N$. But this
assumption is used only to deduce that $L_{A}\left(A\right)=\infty$.
The proof of \cite[Proposition 4.7.]{Alo2020} also uses \cite[Lemma 4.6.]{Alo2020},
which assumes that $A^{\prime}\subseteq A\subseteq\N$, but in fact
the proof of \cite[Lemma 4.6.]{Alo2020} works for any $A^{\prime}\subseteq A\subseteq\Z$.
\end{rem}

\begin{cor}
\label{bdd_two_sided_gaps_for_subsets_but_not_generic_implies_either_N_is_def_or_IP}Let
$A\subseteq\Z$ be infinite, and let $\mathcal{Z}:=\left(\mathbb{Z},+,0,1,A\right)$.
Suppose that every infinite subset of $A$ that is definable in $\mathcal{Z}$
has bounded two-sided gaps. Suppose also that $A$ is not generic
in $\Z$. Then either $\N$ is definable in $\mathcal{Z}$, or the
formula $y-x\in A$ has IP. 
\end{cor}

\begin{proof}
If $A$ is bounded from below, let $m:=\min\left(A\right)$, and let
$A^{\prime}:=A-m$. Then $A^{\prime}\subseteq\N$, $A^{\prime}$ is
infinite, and $\left(\mathbb{Z},+,0,1,A^{\prime}\right)$ is interdefinable
with $\mathcal{Z}$. Moreover, every infinite subset of $A^{\prime}$
that is definable in $\mathcal{Z}$ has bounded two-sided gaps. By
\cref{not_syndetic_implies_L_eq_infty} and \cref{bounded_two_sided_gaps_gives_IP},
either $A^{\prime}$ is generic in $\N$, or the formula $y-x\in A^{\prime}$
has IP. In the former case, $\N$ is definable in $\mathcal{Z}$.
In the latter case, also the formula $y-x\in A$ has IP.

If $A$ is bounded from above, applying the above to $-A$ gives that
either $\N$ is definable in $\mathcal{Z}$, or the formula $y-x\in A$
has IP. 

So suppose that $A$ is unbounded both from below and from above.
If both sets $A\cap\N$, $\left(-A\right)\cap\N$ are generic in $\N$,
then $A$ is in fact generic in $\Z$, a contradiction. 

So at lest one of the sets $A\cap\N$, $\left(-A\right)\cap\N$ is
infinite and not generic in $\N$. By \cref{not_syndetic_implies_L_eq_infty},
$L_{A}\left(A\right)=\infty$. By \cref{bounded_two_sided_gaps_gives_IP},
the formula $y-x\in A$ has IP. 
\end{proof}
\begin{prop}
\label{dp_minimal_implies_bounded_two_sided_gaps}Let $\mathcal{Z}$
be a dp-minimal expansion of $\left(\mathbb{Z},+,0,1\right)$. Then
every infinite subset of $\Z$ definable in $\mathcal{Z}$ has bounded
two-sided gaps.
\end{prop}

\begin{proof}
Let $L$ be the language of $\mathcal{Z}$, and let $L^{\prime}=L\cup\left\{ <\right\} $.
Expand $\mathcal{Z}$ to an $L^{\prime}$-structure $\mathcal{Z}^{\prime}$
by interpreting $<$ as the usual order.

Let $A\subseteq\Z$ be infinite and definable in $\mathcal{Z}$, and
suppose towards a contradiction that it does not have bounded two-sided
gaps. So for all $N<\omega$ there is $a\in A$ such that $A\cap\left[a-N,a+N\right]=\left\{ a\right\} $.
We may assume that we can always find such $a>0$: otherwise, replace
$A$ with $-A$. Note that as $N$ goes to infinity, the minimal positive
$a$ satisfying the above, also goes to infinity. In particular, given
$N$, we can always find such $a$ which is arbitrarily large.

By recursion on $i<\omega$, we construct a strictly increasing sequence
$\left(b_{i}\right)_{i<\omega}$ of positive elements of $A$ such
that for all $i<j<\omega$, $A\cap\left[b_{j}-b_{i},b_{j}+b_{i}\right]=\left\{ b_{j}\right\} $:
Let $b_{0}$ be any positive element of $A$, and, given $b_{i}$,
let $b_{i+1}>b_{i}$ be given by the above for $N=b_{i}$. 

Let $\cM^{\prime}$ be an $\omega_{1}$-saturated elementary extension
of $\mathcal{Z}^{\prime}$, and let $\cM$ be its reduct to $L$.
Denote by $A^{*}$ the interpretation in $\cM$ of the formula defining
$A$ in $\mathcal{Z}$. Let $\left(a_{i}\right)_{i<\omega\cdot2}$
be an $L^{\prime}$-indiscernible sequence locally-based on $\left(b_{i}\right)_{i<\omega}$.
So $\left(a_{i}\right)_{i<\omega\cdot2}$ is a strictly increasing
sequence of positive infinite elements of $A^{*}$, and for all $i<j<\omega\cdot2$,
$A^{*}\cap\left[a_{j}-a_{i},a_{j}+a_{i}\right]=\left\{ a_{j}\right\} $. 

Let $i<j<\omega\cdot2$. Clearly, $a_{j}+a_{i}\in A^{*}+a_{j}$ and
$a_{j}+a_{i}\in A^{*}+a_{i}$. Let $k<j$, $k\neq i$. Then $a_{j}+a_{i}\notin A^{*}+a_{k}$:
Suppose otherwise. So $a_{j}+a_{i}-a_{k}\in A^{*}$. But $a_{j}+a_{i}-a_{k}\in\left[a_{j}-a_{j-1},a_{j}+a_{j-1}\right]$,
therefore $a_{j}+a_{i}-a_{k}=a_{j}$, so $a_{i}=a_{k}$, a contradiction.

Consider the following two families:
\begin{gather*}
\left\{ \left(A^{*}+a_{2i+1}\right)\backslash\left(A^{*}+a_{2i}\right)\,:\,i<\omega\right\} \\
\left\{ \left(A^{*}+a_{\omega+2j+1}\right)\backslash\left(A^{*}+a_{\omega+2j}\right)\,:\,j<\omega\right\} 
\end{gather*}

Since $\mathcal{Z}$ is NIP, both families are $l$-inconsistent for
some $l<\omega$. By the above, for every $i,j<\omega$, $a_{\omega+2j+1}+a_{2i+1}\in\left(A^{*}+a_{2i+1}\right)\backslash\left(A^{*}+a_{2i}\right)$
and $a_{\omega+2j+1}+a_{2i+1}\in\left(A^{*}+a_{\omega+2j+1}\right)\backslash\left(A^{*}+a_{\omega+2j}\right)$.
So this is an inp-pattern of depth $2$, contradicting the dp-minimality
of $\mathcal{Z}$.
\end{proof}
We also need the following:
\begin{fact}[{\cite[Corollary 2.20]{DolichGoodrick2017}}]
\label{no_strongly_dependent_expansions_of_the_order}Suppose
that $\cZ$ is a strong expansion of $\left(\mathbb{Z},+,0,1,<\right)$.
Then $\mathcal{Z}$ is interdefinable with $\left(\mathbb{Z},+,0,1,<\right)$.
\end{fact}

Since every strongly-dependent theory is strong, this fact implies
that every proper expansion of $\left(\mathbb{Z},+,0,1,<\right)$
has $\mbox{dp-rank}\ge\omega$.

\begin{proof}[Proof of \cref{main_theorem_genericity}]
Let $A\subseteq\Z$ be infinite and definable in $\mathcal{Z}$,
and suppose towards a contradiction that it is not generic in $\Z$.
By \cref{dp_minimal_implies_bounded_two_sided_gaps}, every infinite
subset of $A$ that is definable in $\mathcal{Z}$ has bounded two-sided
gaps. Since $\mathcal{Z}$ is NIP, by \cref{bdd_two_sided_gaps_for_subsets_but_not_generic_implies_either_N_is_def_or_IP}
we get that $\N$ is definable in $\mathcal{Z}$, so $\cZ$ is a dp-minimal
expansion of $\left(\mathbb{Z},+,0,1,<\right)$. By \cref{no_strongly_dependent_expansions_of_the_order},
$\mathcal{Z}$ is interdefinable with $\left(\mathbb{Z},+,0,1,<\right)$,
a contradiction.
\end{proof}
\begin{rem}
Note that \cref{main_theorem_genericity} is only about the structure
$\mathcal{Z}$ itself, not the monster. It implies that every infinite
subset of the monster which is definable over $\Z$ is generic in
the monster, but other infinite definable subsets of the monster might
not be generic.

In fact, it can be shown that, in general, if $G$ is an NIP group,
and every infinite definable subset of the monster is generic, then
$G$ must be stable.
\end{rem}

\section{Obtaining a homomorphism to \texorpdfstring{$\protect\R\big/\protect\Z$}{R/Z}}

In this section, $T$ is any theory, and $G$ is a group $\emptyset$-definable
in $T$ for which $G^{00}$ exists (and hence so does $G^{0}$).
\begin{lem}
\label{G00_neq_G0_equiv_nbrhd_with_no_closed_finite_idx_subgroups}$G^{00}\neq G^{0}$
if and only if there is an open neighborhood $U$ of the identity
in $G\big/G^{00}$ which does not contain any closed subgroups of
finite index.
\end{lem}

\begin{proof}
Denote by $\pi:G\rightarrow G\big/G^{00}$ the quotient map. Suppose
that $G^{00}\neq G^{0}$, so $\pi\left(G^{0}\right)\neq\left\{ 1\right\} $.
Let $1\neq h\in\pi\left(G^{0}\right)$. Since $G\big/G^{00}$ is Hausdorff,
there are open neighborhoods $U$,$V$ of $1$,$h$ respectively,
which are disjoint. So $\pi\left(G^{0}\right)\nsubseteq U$.

Suppose towarda a contradiction that $H^{\prime}\subseteq U$ for
some closed subgroup $H^{\prime}$ of $G\big/G^{00}$ of finite index,
and let $H:=\pi^{-1}\left(H^{\prime}\right)$. So $H$ is of finite
index in $G$. Since $G\big/G^{00}$ is compact, $H^{\prime}$ is
also open, therefore $H$ is definable over some (any) small model
$\cM$. Therefore $G^{0}\subseteq H$, so $\pi\left(G^{0}\right)\subseteq H^{\prime}\subseteq U$,
a contradiction.

On the other hand, suppose that $G^{00}=G^{0}$, so $\pi\left(G^{0}\right)=\left\{ 1\right\} $.
Let $U$ be an open neighborhood of $1$ in $G\big/G^{00}$. So $G^{0}\subseteq\pi^{-1}\left(U\right)$,
and $\pi^{-1}\left(U\right)$ is $\bigvee$-definable over some (any)
small model $\cM$. By the definition of $G^{0}$ and the saturation
of the monster, there is an $\cM$-definable subgroup $H$ of $G$
of finite index such that $H\subseteq\pi^{-1}\left(U\right)$. So
$\pi\left(H\right)\subseteq U$. Since $G^{00}\subseteq H$, $\pi^{-1}\left(\pi\left(H\right)\right)=HG^{00}=H$,
and so $\pi\left(H\right)$ is clopen and of finite index.
\end{proof}
We need the following two well-known facts:
\begin{fact}[{\cite[Gleason-Yamabe theorem for abelian groups]{Tao2014}}]
\label{Gleason_Yamabe_abelian}Let $G$ be a locally compact
abelian Hausdorff group, and let $U$ be a neighborhood of the identity.
Then there is a compact normal subgroup $K$ of $G$ contained in
$U$ such that $G\big/K$ is isomorphic to a Lie group.
\end{fact}

\begin{fact}
\label{structure_of_compact_abelian_Lie_groups}Every compact
abelian Lie group is isomorphic to $\left(\R\big/\Z\right)^{d}\times F$
for some $d\in\omega$ and $F$ a finite abelian group.
\end{fact}

\begin{lem}
\label{homomorphism_to_the_circle}Suppose that $G$ is abelian
and $G^{00}\neq G^{0}$. Then there is a surjective group homomorphism
$h:G\twoheadrightarrow\R\big/\Z$ such that for any small model $\cM$:
\begin{enumerate}
\item A set $C\subseteq\R\big/\Z$ is closed if and only if $h^{-1}\left(C\right)$
is $\cM$-type-definable.
\item $h\left(\cM\right)$ is dense.
\end{enumerate}
\end{lem}

\begin{proof}
By \cref{G00_neq_G0_equiv_nbrhd_with_no_closed_finite_idx_subgroups}
there is an open neighborhood $U$ of the identity in $\bar{G}:=G\big/G^{00}$
which does not contain any closed subgroups of finite index. By \cref{Gleason_Yamabe_abelian}
there is a compact (so also closed) normal subgroup $K$ of $\bar{G}$
contained in $U$ such that $\bar{G}\big/K$ is isomorphic to a Lie
group. By \cref{structure_of_compact_abelian_Lie_groups}, $\bar{G}\big/K$
is of the form $\left(\R\big/\Z\right)^{d}\times F$ for some $d\in\omega$
and $F$ a finite abelian group. By the choice of $U$, $K$ has infinite
index in $\bar{G}$, so $d\ge1$. 

Let $p:\bar{G}\big/K\cong\left(\R\big/\Z\right)^{d}\times F\to\R\big/\Z$
be the projection on the first coordinate, and denote also by $\pi:G\rightarrow\bar{G}$
and $\rho:\bar{G}\to\bar{G}\big/K$ the quotient maps. Let $h:=p\circ\rho\circ\pi$.
Since $\rho$ and $p$ are topological quotient maps, and since the
topology on $\bar{G}$ is the logic topology, $h$ satisfies $\left(1\right)$.

$\left(2\right)$ follows from $\left(1\right)$: Let $\emptyset\neq O\subseteq\R\big/\Z$
be open. So $h^{-1}\left(O\right)\neq\emptyset$ is $\bigvee$-definable
over $\cM$, therefore $h^{-1}\left(O\right)\cap\cM\neq\emptyset$,
so $O\cap h\left(\cM\right)\neq\emptyset$.
\end{proof}
\begin{rem}
\label{if_P_is_type_def_then_its_img_under_h_is_closed}In the
context of \cref{homomorphism_to_the_circle}, in particular, $\ker\left(h\right)$
is type-definable over any small model. Therefore, if $P\subseteq G$
is type-definable (over any small set), then $h^{-1}\left(h\left(P\right)\right)=P+\ker\left(h\right)$
is also type-definable, and hence $h\left(P\right)$ is closed in
$\R\big/\Z$.
\end{rem}

\section{Recovering the cyclic order}

In this section we prove \cref{main_theorem_cyclic_order}. Let
$\mathcal{Z}$ be a dp-minimal expansion of $\left(\mathbb{Z},+,0,1\right)$
with monster model $G$, and suppose that $G^{00}\neq G^{0}$. Let
$h:G\twoheadrightarrow\R\big/\Z$ be a homomorphism as given by \cref{homomorphism_to_the_circle}.
Let $\alpha\in\R$ be such that $h\left(1\right)=\alpha+\Z$. Since
$h\left(\Z\right)$ is dense, $\alpha\in\R\backslash\Q$. Note that
$h\restriction_{\Z}$ is injective. We will show that $C_{\alpha}$
is definable in $\cZ$.
\begin{notation}
Let $q:\R\to\R\big/\Z$ be the quotient map, and let $\iota:=\left(q\restriction_{\left[-\frac{1}{2},\frac{1}{2}\right)}\right)^{-1}:\R\big/\Z\to\left[-\frac{1}{2},\frac{1}{2}\right)$.
Note that $q\restriction_{\left(-\frac{1}{2},\frac{1}{2}\right)}$
and $\iota\restriction_{\left(\R/\Z\right)\backslash\left\{ q\left(-\frac{1}{2}\right)\right\} }$
are homeomorphisms.

When it's clear from the context, we identify intervals $I\subseteq\R$
with their image $q\left(I\right)$. So, for example, $h^{-1}\left(\left[-\frac{1}{4},\frac{1}{4}\right]\right)$
means $h^{-1}\left(q\left(\left[-\frac{1}{4},\frac{1}{4}\right]\right)\right)$.
\end{notation}

\begin{rem}
\label{restricted_additivity_of_iota}If $s,t\in\R\big/\Z$ and
$\i s+\i t\in\left[-\frac{1}{2},\frac{1}{2}\right)$ then $\i{s+t}=\i{q\left(\i s\right)+q\left(\i t\right)}=\i{q\left(\i s+\i t\right)}=\i s+\i t$.
In particular, this is true for $s,t\in q\left(\left(-\frac{1}{4},\frac{1}{4}\right)\right)$.

We also have $\i{-s}=-\i s$ for all $s\in\R\big/\Z$.
\end{rem}

\begin{notation}
\label{approximate_balls_notation}Let $u_{1}<v_{1}\le v_{2}<u_{2}$
in $\R$. Then $h^{-1}\left(\left[v_{1},v_{2}\right]\right)\subseteq h^{-1}\left(\left(u_{1},u_{2}\right)\right)$,
where $h^{-1}\left(\left[v_{1},v_{2}\right]\right)$ is type-definable
over $\Z$ and $h^{-1}\left(\left(u_{1},u_{2}\right)\right)$ is $\bigvee$-definable
over $\Z$. By saturation, there is a $\Z$-definable set $h^{-1}\left(\left[v_{1},v_{2}\right]\right)\subseteq B\subseteq h^{-1}\left(\left(u_{1},u_{2}\right)\right)$.
We fix one such set and denote it by $B_{u_{1},v_{1},v_{2},u_{2}}$. 

For $u\in\R$ and $0\le r_{1}<r_{2}$, we denote $B_{u,r_{1},r_{2}}:=B_{u-r_{2},u-r_{1},u+r_{1},u+r_{2}}$.
For $r>0$ we also denote $B_{u,r}:=B_{u,\frac{r}{2},r}$.

We may assume that the sets of the form $B_{0,r_{1},r_{2}}$ (and
$B_{0,r}$) are symmetric, by replacing them with $B_{0,r_{1},r_{2}}\cup\left(-B_{0,r_{1},r_{2}}\right)$.
\end{notation}

\begin{rem}
\label{genericity_in_the_monster_with_translations_in_Z}Since
$G^{00}\neq G^{0}$, $\mathcal{Z}$ is not interdefinable with $\left(\mathbb{Z},+,0,1,<\right)$,
so by \cref{main_theorem_genericity}, every infinite subset of
$\Z$ definable in $\mathcal{Z}$ is generic in $\Z$. By elementarity,
for every infinite subset $A$ of $G$ which is definable over $\Z$,
there is a finite subset $F\subseteq\Z$ such that $A+F=G$.
\end{rem}

\begin{lem}
\label{local_genericity}Let $0<r\in\R$, and let $E$ be an infinite
$\Z$-definable subset of $G$ such that $h\left(E\right)\subseteq\left(-\frac{r}{2},\frac{r}{2}\right)$.
Then for every $0<u<\frac{1}{2}$ there is a finite $F\subseteq\Z$
such that $h\left(F\right)\subseteq\left[-u-\frac{r}{2},u+\frac{r}{2}\right]$
and 
\[
h^{-1}\left(\left[-u,u\right]\right)\subseteq E+F\subseteq h^{-1}\left(\left(-u-r,u+r\right)\right)
\]
\end{lem}

\begin{proof}
Since $E$ is infinite and $\Z$-definable, by \cref{genericity_in_the_monster_with_translations_in_Z}
there is a finite $F^{\prime}\subseteq\Z$ such that $E+F^{\prime}=G$.
Note that for $c\in G$, if $h\left(c\right)\notin\left[-u-\frac{r}{2},u+\frac{r}{2}\right]$
then $h\left(E+c\right)\cap\left[-u,u\right]=\emptyset$. So $F:=F^{\prime}\cap h^{-1}\left(\left[-u-\frac{r}{2},u+\frac{r}{2}\right]\right)$
is as required.
\end{proof}
\begin{lem}
\label{isolated_pt_of_img_of_Z_def_set_comes_frm_sngle_elmnt_of_Z}Let
$A\subseteq G$ be $\Z$-definable, and let $s\in h\left(A\right)$.
If $s$ is an isolated point of $h\left(A\right)$, then there is
$b\in\Z$ such that $A\cap h^{-1}\left(s\right)=\left\{ b\right\} $.
\end{lem}

\begin{proof}
Since $h\restriction_{\Z}$ is injective, $\left\vert \Z\cap h^{-1}\left(s\right)\right\vert \le1$.
If $\left\vert \Z\cap h^{-1}\left(s\right)\right\vert =1$, denote
by $b$ the single element of $\Z\cap h^{-1}\left(s\right)$, and
otherwise, denote $b:=0$. Since $s$ is an isolated point of $h\left(A\right)$,
there is $0<r\in\R$ such that $h\left(A\right)\cap\left(\i s-r,\i s+r\right)=\left\{ s\right\} $.
Let $B:=B_{\i s,r}$ as in \cref{approximate_balls_notation}.

Suppose towards a contradiction that $A\cap h^{-1}\left(s\right)\neq\left\{ b\right\} $
(in particular, this holds if $\Z\cap h^{-1}\left(s\right)=\emptyset$).
Since $s\in h\left(A\right)$, $A\cap h^{-1}\left(s\right)\neq\emptyset$,
so let $d\in A\cap h^{-1}\left(s\right)$ such that $d\neq b$. In
particular, $d\in h^{-1}\left(\left[\i s-\frac{r}{2},\i s+\frac{r}{2}\right]\right)\subseteq B$.
So $d\in\left(A\cap B\right)\backslash\left\{ b\right\} $. But $\left(A\cap B\right)\backslash\left\{ b\right\} $
is $\Z$-definable, so, by elementarity, there exists $a\in\left(\left(A\cap B\right)\backslash\left\{ b\right\} \right)\cap\Z$.
Therefore $h\left(a\right)\in h\left(A\cap B\right)\subseteq h\left(A\right)\cap h\left(B\right)\subseteq h\left(A\right)\cap\left(\i s-r,\i s+r\right)=\left\{ s\right\} $,
so $a\in\Z\cap h^{-1}\left(s\right)\subseteq\left\{ b\right\} $,
a contradiction.
\end{proof}
\begin{lem}
\label{local_genericity_around_an_accumulation_point}Let $A\subseteq G$
be $\Z$-definable, let $0<d\in\R$, and let $p\in\left(-\frac{d}{2},\frac{d}{2}\right)$
be an accumulation point of $h\left(A\right)$. Then there is a finite
subset $F\subseteq\Z$ such that $h\left(F\right)\subseteq\left(-d,d\right)$,
$0\notin F$, and $h^{-1}\left(\left(-\frac{d}{2},\frac{d}{2}\right)\right)\subseteq A+F$.
\end{lem}

\begin{proof}
Let $d^{\prime}\in\left(0,\frac{d}{2}\right)$ be such that $p\in\left(-d^{\prime},d^{\prime}\right)$,
and let $B:=B_{0,d^{\prime},\frac{d}{2}}$ as in \cref{approximate_balls_notation}.
Since $p\in\left(-d^{\prime},d^{\prime}\right)$ is an accumulation
point of $h\left(A\right)$, $h\left(A\right)\cap\left(-d^{\prime},d^{\prime}\right)$
is infinite. But $h\left(A\cap B\right)\supseteq h\left(A\cap h^{-1}\left(\left(-d^{\prime},d^{\prime}\right)\right)\right)=h\left(A\right)\cap\left(-d^{\prime},d^{\prime}\right)$,
so $h\left(A\cap B\right)$, and hence $A\cap B$, are infinite as
well. 

Since $A$ and $B$ are $\Z$-definable, so is $A\cap B$. So by elementarity,
$A\cap B\cap\Z$ is infinite. By \cref{genericity_in_the_monster_with_translations_in_Z}
there is a finite $F^{\prime}\subseteq\Z$ such that $A\cap B\cap\Z+F^{\prime}=\Z$.
Let $m\in\Z\backslash F^{\prime}$, and let $F^{\prime\prime}:=F^{\prime}-m$.
So $0\notin F^{\prime\prime}$ and $A\cap B\cap\Z+F^{\prime\prime}=\Z-m=\Z$.
By elementarity, $A\cap B+F^{\prime\prime}=G$.

Let $F:=\left\{ c\in F^{\prime\prime}\,:\,\left(A\cap B+c\right)\cap h^{-1}\left(\left(-\frac{d}{2},\frac{d}{2}\right)\right)\neq\emptyset\right\} $.
So $A+F\supseteq A\cap B+F\supseteq h^{-1}\left(\left(-\frac{d}{2},\frac{d}{2}\right)\right)$.
If $d>\frac{1}{2}$ then clearly $h\left(F\right)\subseteq\left(-d,d\right)$,
so suppose $d\le\frac{1}{2}$. Let $c\in\Z$. By the choice of $B$,
$h\left(A\cap B+c\right)\subseteq h\left(B+c\right)=h\left(B\right)+h\left(c\right)\subseteq\left(-\frac{d}{2}+\i{h\left(c\right)},\frac{d}{2}+\i{h\left(c\right)}\right)$,
so $A\cap B+c\subseteq h^{-1}\left(\left(-\frac{d}{2}+\i{h\left(c\right)},\frac{d}{2}+\i{h\left(c\right)}\right)\right)$.
Therefore, for $c\in F$ we get $h^{-1}\left(\left(-\frac{d}{2},\frac{d}{2}\right)\right)\cap h^{-1}\left(\left(-\frac{d}{2}+\i{h\left(c\right)},\frac{d}{2}+\i{h\left(c\right)}\right)\right)\neq\emptyset$,
so $\left(-\frac{d}{2},\frac{d}{2}\right)\cap\left(-\frac{d}{2}+\i{h\left(c\right)},\frac{d}{2}+\i{h\left(c\right)}\right)\neq\emptyset$.
Recall that these intervals actually denote their images under the
quotient map $q:\R\to\R\big/\Z$, i.e., we only have $q\left(\left(-\frac{d}{2}+\i{h\left(c\right)},\frac{d}{2}+\i{h\left(c\right)}\right)\right)\cap q\left(\left(-\frac{d}{2},\frac{d}{2}\right)\right)\neq\emptyset$.
But since $d\le\frac{1}{2}$, and since $\i{h\left(c\right)}\in\left[-\frac{1}{2},\frac{1}{2}\right)$,
this implies that $-d<\i{h\left(c\right)}<d$, so $h\left(c\right)\in\left(-d,d\right)$.
This shows that $h\left(F\right)\subseteq\left(-d,d\right)$.
\end{proof}
\begin{cor}
\label{existence_of_small_translation_with_infinite_intersection}For
every infinite $\Z$-definable set $B\subseteq G$ and every $0<r\in\R$,
there is $b\in\Z$ with $0<\i{h\left(b\right)}<r$ such that $B\cap\left(B+b\right)$
is infinite.
\end{cor}

\begin{proof}
We may assume $r<\frac{1}{2}$. By elementarity, $B\cap\Z$ is infinite.
Since $h\restriction_{\Z}$ is injective, $h\left(B\right)\supseteq h\left(B\cap\Z\right)$
is infinite, and hence has an accumulation point $p$. Since $h\left(\Z\right)$
is dense, there is $m\in\Z$ such that $h\left(m\right)\in\left(p-\frac{r}{2},p+\frac{r}{2}\right)$.
Let $B^{\prime}:=B-m$. So $p^{\prime}:=p-h\left(m\right)\in\left(-\frac{r}{2},\frac{r}{2}\right)$
is an accumulation point of $h\left(B^{\prime}\right)=h\left(B\right)-h\left(m\right)$. 

By \cref{local_genericity_around_an_accumulation_point} there is
a finite subset $F\subseteq\Z$ such that $h\left(F\right)\subseteq\left(-r,r\right)$,
$0\notin F$, and $h^{-1}\left(\left(-\frac{r}{2},\frac{r}{2}\right)\right)\subseteq B^{\prime}+F$.
So $h\left(B^{\prime}\cap\left(B^{\prime}+F\right)\right)\supseteq h\left(B^{\prime}\cap h^{-1}\left(\left(-\frac{r}{2},\frac{r}{2}\right)\right)\right)=h\left(B^{\prime}\right)\cap\left(-\frac{r}{2},\frac{r}{2}\right)$.
Since $p^{\prime}\in\left(-\frac{r}{2},\frac{r}{2}\right)$ is an
accumulation point of $h\left(B^{\prime}\right)$, we get that $h\left(B^{\prime}\cap\left(B^{\prime}+F\right)\right)$,
and hence $B^{\prime}\cap\left(B^{\prime}+F\right)$, are infinite.
But $B^{\prime}\cap\left(B^{\prime}+F\right)=B\cap\left(B+F\right)-m$,
so $B\cap\left(B+F\right)$ is infinite.

Since $F$ is finite, there is $b\in F$ such that $B\cap\left(B+b\right)$
is infinite. So $h\left(b\right)\in\left(-r,r\right)$. Since $0\notin F\subseteq\Z$
and $h\restriction_{\Z}$ is injective, $h\left(b\right)\neq0$. Note
that $B\cap\left(B-b\right)=B\cap\left(B+b\right)-b$ is also infinite,
so by replacing $b$ with $-b$ if necessary, we may assume that $h\left(b\right)\in\left(0,r\right)$.
Since $r<\frac{1}{2}$ we get $0<\i{h\left(b\right)}<r$.
\end{proof}
\begin{lem}
\label{uniformly_def_family_of_approximate_balls}There is a uniformly
$\Z$-definable family $\left\{ E_{r}\,:\,0<r\in\R\right\} $ of infinite
subsets of $G$, such that for each $r$, $h\left(E_{r}\right)\subseteq\left(-\frac{r}{2},\frac{r}{2}\right)$.
\end{lem}

\begin{proof}
Let $A:=B_{0,\frac{1}{32}}$. By \cref{if_P_is_type_def_then_its_img_under_h_is_closed},
$h\left(A\right)$ is closed, so $C:=\i{h\left(A\right)}$ is closed
as well. Let $s:=\min\left(C\right)$ and $t:=\max\left(C\right)$.
Let $C^{\prime}\subseteq C$ be the set of accumulation points of
$C$, which is closed as well, and let $s^{\prime}:=\min\left(C^{\prime}\right)$
and $t^{\prime}:=\max\left(C^{\prime}\right)$. Let $S:=C\cap\left[s,s^{\prime}\right)$
and $T:=C\cap\left(t^{\prime},t\right]$. By definition, for every
$u\in\left[s,s^{\prime}\right)$, the interval $\left[s,u\right]$
contains no accumulation points of $C$, so $C\cap\left[s,u\right]$
is finite. Therefore, either $S$ is finite, or $S$ has order-type
$\omega$ with $\sup\left(S\right)=s^{\prime}$. Similarly, either
$T$ is finite, or $T$ has order-type $\left(\omega,>\right)$ with
$\inf\left(T\right)=T^{\prime}$. Denote $S=\left\{ s_{i}\right\} _{i<\alpha}$,
$T=\left\{ t_{i}\right\} _{i<\beta}$, where $\alpha,\beta\in\omega+1$
and for every $i<j$, $s_{i}<s_{j}$ and $t_{i}>t_{j}$. Note that
if $S=\emptyset$ (resp. $T=\emptyset$) then $s=s^{\prime}$ (resp.
$t=t^{\prime}$), and otherwise, $s=s_{0}$ (resp. $t=t_{0}$).
\begin{claim}
\label{claim_inside_uniformly_def_family_of_approximate_balls}For
every $0<r\in\R$ there are $b_{1},b_{2}\in\Z$ such that:
\begin{enumerate}
\item $\left(A+b_{1}\right)\cap\left(A+b_{2}\right)$ is infinite,
\item $\i{h\left(b_{1}\right)}<0<\i{h\left(b_{2}\right)}$, and
\item $-\frac{r}{2}<s^{\prime}+\i{h\left(b_{2}\right)}<0<t^{\prime}+\i{h\left(b_{1}\right)}<\frac{r}{2}$.
\end{enumerate}
\end{claim}

\begin{proof}[Proof of Claim]
Let $0<d<\min\left(t^{\prime},-s^{\prime},\frac{r}{2},\frac{1}{16}\right)$,
and let $b_{2}^{\prime}\in\Z$ be such that $-\frac{d}{4}<s^{\prime}+\i{h\left(b_{2}^{\prime}\right)}<0$
(which exists since $\i{h\left(\Z\right)}$ is dense in $\left(-\frac{1}{2},\frac{1}{2}\right)$
and $s^{\prime}<\frac{1}{2}$). By definition of $s^{\prime}$, $q\left(s^{\prime}+\i{h\left(b_{2}^{\prime}\right)}\right)=q\left(s^{\prime}\right)+h\left(b_{2}^{\prime}\right)$
is an accumulation point of $q\left(C+\i{h\left(b_{2}^{\prime}\right)}\right)=q\left(C\right)+h\left(b_{2}^{\prime}\right)=h\left(A\right)+h\left(b_{2}^{\prime}\right)=h\left(A+b_{2}^{\prime}\right)$,
so by \cref{local_genericity_around_an_accumulation_point} there
is a finite subset $F\subseteq\Z$ such that $h\left(F\right)\subseteq\left(-d,d\right)$,
$0\notin F$, and $h^{-1}\left(\left(-\frac{d}{2},\frac{d}{2}\right)\right)\subseteq A+b_{2}^{\prime}+F$.

Let $b_{1}^{\prime}\in\Z$ be such that $0<t^{\prime}+\i{h\left(b_{1}^{\prime}\right)}<\frac{d}{2}$
and $t^{\prime}+\i{h\left(b_{1}^{\prime}\right)}\notin\i{h\left(F\right)}+s^{\prime}+\i{h\left(b_{2}^{\prime}\right)}$
(which exists since $\i{h\left(\Z\right)}$ is dense in $\left(-\frac{1}{2},\frac{1}{2}\right)$,
$t^{\prime}>\frac{1}{2}$, and $F$ is finite). By definition of $t^{\prime}$,
$q\left(t^{\prime}+\i{h\left(b_{1}^{\prime}\right)}\right)$ is an
accumulation point of $q\left(C+\i{h\left(b_{1}^{\prime}\right)}\right)=h\left(A+b_{1}^{\prime}\right)$,
so $h\left(A+b_{1}^{\prime}\right)\cap\left(-\frac{d}{2},\frac{d}{2}\right)$
is infinite. But $h\left(\left(A+b_{1}^{\prime}\right)\cap\left(A+b_{2}^{\prime}+F\right)\right)\supseteq h\left(\left(A+b_{1}^{\prime}\right)\cap h^{-1}\left(\left(-\frac{d}{2},\frac{d}{2}\right)\right)\right)=h\left(A+b_{1}^{\prime}\right)\cap\left(-\frac{d}{2},\frac{d}{2}\right)$,
so $h\left(\left(A+b_{1}^{\prime}\right)\cap\left(A+b_{2}^{\prime}+F\right)\right)$,
and hence $\left(A+b_{1}^{\prime}\right)\cap\left(A+b_{2}^{\prime}+F\right)$,
are infinite as well. 

Since $F$ is finite, there is $c\in F$ such that $\left(A+b_{1}^{\prime}\right)\cap\left(A+b_{2}^{\prime}+c\right)$
is infinite. By the choice of $b_{1}^{\prime}$, $t^{\prime}+\i{h\left(b_{1}^{\prime}\right)}\neq\i{h\left(c\right)}+s^{\prime}+\i{h\left(b_{2}^{\prime}\right)}$.
Suppose towards a contradiction that $t^{\prime}+\i{h\left(b_{1}^{\prime}\right)}<\i{h\left(c\right)}+s^{\prime}+\i{h\left(b_{2}^{\prime}\right)}$,
and let $u\in\R$ be such that $t^{\prime}+\i{h\left(b_{1}^{\prime}\right)}<u<\i{h\left(c\right)}+s^{\prime}+\i{h\left(b_{2}^{\prime}\right)}$.
By the definitions of $s$ and $t$, $\i{h\left(\left(A+b_{1}^{\prime}\right)\cap\left(A+b_{2}^{\prime}+c\right)\right)}\subseteq\left[\i{h\left(c\right)}+s+\i{h\left(b_{2}^{\prime}\right)},t+\i{h\left(b_{1}^{\prime}\right)}\right]$,
therefore 
\begin{gather*}
\i{h\left(\left(A+b_{1}^{\prime}\right)\cap\left(A+b_{2}^{\prime}+c\right)\right)}=\\
\i{h\left(\left(A+b_{1}^{\prime}\right)\cap\left(A+b_{2}^{\prime}+c\right)\right)}\cap\left(\left[\i{h\left(c\right)}+s+\i{h\left(b_{2}^{\prime}\right)},u\right]\cup\left[u,t+\i{h\left(b_{1}^{\prime}\right)}\right]\right)\subseteq\\
\left(\i{h\left(A+b_{1}^{\prime}\right)}\cap\left[u,t+\i{h\left(b_{1}^{\prime}\right)}\right]\right)\cup\left(\i{h\left(A+b_{2}^{\prime}+c\right)}\cap\left[\i{h\left(c\right)}+s+\i{h\left(b_{2}^{\prime}\right)},u\right]\right)
\end{gather*}
By definition of $t^{\prime}$, and since $t^{\prime}+\i{h\left(b_{1}^{\prime}\right)}<u$,
the set $\i{h\left(A+b_{1}^{\prime}\right)}\cap\left[u,t+\i{h\left(b_{1}^{\prime}\right)}\right]=\left(\i{h\left(A\right)}+\i{h\left(b_{1}^{\prime}\right)}\right)\cap\left(\left[u-\i{h\left(b_{1}^{\prime}\right)},t\right]+\i{h\left(b_{1}^{\prime}\right)}\right)=\i{h\left(A\right)}\cap\left[u-\i{h\left(b_{1}^{\prime}\right)},t\right]+\i{h\left(b_{1}^{\prime}\right)}$
is finite. Similarly, by definition of $s^{\prime}$, and since $u<\i{h\left(c\right)}+s^{\prime}+\i{h\left(b_{2}^{\prime}\right)}$,
the set $\i{h\left(A+b_{2}^{\prime}+c\right)}\cap\left[\i{h\left(c\right)}+s+\i{h\left(b_{2}^{\prime}\right)},u\right]$
is finite. So $\i{h\left(\left(A+b_{1}^{\prime}\right)\cap\left(A+b_{2}^{\prime}+c\right)\right)}$,
and hence $\i{h\left(\left(A+b_{1}^{\prime}\right)\cap\left(A+b_{2}^{\prime}+c\right)\cap\Z\right)}$,
are finite as well. Since $\iota$ and $h\restriction_{\Z}$ are injective,
$\left(A+b_{1}^{\prime}\right)\cap\left(A+b_{2}^{\prime}+c\right)\cap\Z$
is finite. Since $b_{1}^{\prime},b_{2}^{\prime},c\in\Z$ and $A$
is $\Z$-definable, $\left(A+b_{1}^{\prime}\right)\cap\left(A+b_{2}^{\prime}+c\right)$
is finite, a contradiction. So $t^{\prime}+\i{h\left(b_{1}^{\prime}\right)}>\i{h\left(c\right)}+s^{\prime}+\i{h\left(b_{2}^{\prime}\right)}$.

Since $s^{\prime}+\i{h\left(b_{2}^{\prime}\right)}>-\frac{d}{4}$,
$t^{\prime}+\i{h\left(b_{1}^{\prime}\right)}<\frac{d}{2}$, and $\i{h\left(c\right)}>-d$,
we have $0<\left(t^{\prime}+\i{h\left(b_{1}^{\prime}\right)}\right)-\left(\i{h\left(c\right)}+s^{\prime}+\i{h\left(b_{2}^{\prime}\right)}\right)<\frac{7d}{4}$,
hence there exists $c^{\prime}\in\Z$ such that $-d<-\frac{7d}{8}<\i{h\left(c\right)}+s^{\prime}+\i{h\left(b_{2}^{\prime}\right)}+\i{h\left(c^{\prime}\right)}<0<t^{\prime}+\i{h\left(b_{1}^{\prime}\right)}+\i{h\left(c^{\prime}\right)}<\frac{7d}{8}<d$.
Let $b_{1}:=b_{1}^{\prime}+c^{\prime}$, $b_{2}:=b_{2}^{\prime}+c+c^{\prime}$.
So $0<t^{\prime}+\i{h\left(b_{1}\right)}<d<t^{\prime}$ and $0>s^{\prime}+\i{h\left(b_{2}\right)}>-d>s^{\prime}$,
so in particular, $\i{h\left(b_{1}\right)}<0<\i{h\left(b_{2}\right)}$.
Moreover, since $d<\frac{r}{2}$ we have $-\frac{r}{2}<s^{\prime}+\i{h\left(b_{2}\right)}<0<t^{\prime}+\i{h\left(b_{1}\right)}<\frac{r}{2}$.
Finally, $\left(A+b_{1}\right)\cap\left(A+b_{2}\right)=\left(A+b_{1}^{\prime}\right)\cap\left(A+b_{2}^{\prime}+c\right)+c^{\prime}$
is infinite.
\end{proof}
Fix $0<r\in\R$ and let $b_{1},b_{2}\in\Z$ be as in \cref{claim_inside_uniformly_def_family_of_approximate_balls}
for this $r$. Let $E_{r}^{\prime}:=\left(A+b_{1}\right)\cap\left(A+b_{2}\right)$.
Note that $\i{h\left(E_{r}^{\prime}\right)}\subseteq\left[s^{\prime}+\i{h\left(b_{2}\right)},t^{\prime}+\i{h\left(b_{1}\right)}\right]\cup\left(S+\i{h\left(b_{2}\right)}\right)\cup\left(T+\i{h\left(b_{1}\right)}\right)$.

Let $0<r^{\prime}\in\R$ be such that $-\frac{r}{2}<-r^{\prime}<s^{\prime}+\i{h\left(b_{2}\right)}<0<t^{\prime}+\i{h\left(b_{1}\right)}<r^{\prime}<\frac{r}{2}$.
Let $I_{1}:=\left\{ i<\alpha\,:\,s_{i}+\i{h\left(b_{2}\right)}\le-r^{\prime}\right\} $
and $I_{2}:=\left\{ i<\beta\,:\,t_{i}+\i{h\left(b_{1}\right)}\ge r^{\prime}\right\} $.
By the definitions of $s^{\prime}$ and $t^{\prime}$, $I_{1}$ and
$I_{2}$ are finite. Let $0<r^{\prime\prime}\in\R$ be such that:
\begin{enumerate}
\item For each $0<i\in I_{1}$, $r^{\prime\prime}<s_{i}-s_{i-1}$, 
\item For each $0<i\in I_{2}$, $r^{\prime\prime}<t_{i-1}-t_{i}$,
\item $r^{\prime}+r^{\prime\prime}<\frac{r}{2}$.
\end{enumerate}
By \cref{existence_of_small_translation_with_infinite_intersection},
there is $c\in\Z$ with $0<\i{h\left(c\right)}<r^{\prime\prime}$
such that $E_{r}:=E_{r}^{\prime}\cap\left(E_{r}^{\prime}+c\right)$
is infinite. Note that the family $\left\{ E_{r}\,:\,0<r\in\R\right\} $
is uniformly $\Z$-definable. We have $\i{h\left(E_{r}\right)}\subseteq\left(-r^{\prime},r^{\prime}+r^{\prime\prime}\right)\cup\left\{ s_{i}+\i{h\left(b_{2}\right)}\,:\,i\in I_{1}\right\} \cup\left\{ t_{i}+\i{h\left(b_{1}\right)}+\i{h\left(c\right)}\,:\,i\in I_{2}\right\} $.

Let $i\in I_{1}$. Suppose towards a contradiction that $s_{i}+\i{h\left(b_{2}\right)}\in\i{h\left(E_{r}\right)}$.
So in particular, $s_{i}+\i{h\left(b_{2}\right)}-\i{h\left(c\right)}\in\i{h\left(E_{r}^{\prime}\right)}$.
Since $\i{h\left(c\right)}>0$, $s_{i}+\i{h\left(b_{2}\right)}-\i{h\left(c\right)}\le-r^{\prime}$,
so there is $j\in I_{1}$, $j<i$, such that $s_{i}+\i{h\left(b_{2}\right)}-\i{h\left(c\right)}=s_{j}+\i{h\left(b_{2}\right)}$.
So $s_{j+1}-s_{j}\le s_{i}-s_{j}=\i{h\left(c\right)}<r^{\prime\prime}$,
a contradiction. Therefore $s_{i}+\i{h\left(b_{2}\right)}\notin\i{h\left(E_{r}\right)}$.

Let $i\in I_{2}$. Suppose towards a contradiction that $t_{i}+\i{h\left(b_{1}\right)}+\i{h\left(c\right)}\in\i{h\left(E_{r}\right)}$.
So in particular, $t_{i}+\i{h\left(b_{1}\right)}+\i{h\left(c\right)}\in\i{h\left(E_{r}^{\prime}\right)}$.
Since $\i{h\left(c\right)}>0$, $t_{i}+\i{h\left(b_{1}\right)}+\i{h\left(c\right)}\ge r^{\prime}$,
so there is $j\in I_{2}$, $j<i$, such that $t_{i}+\i{h\left(b_{1}\right)}+\i{h\left(c\right)}=t_{j}+\i{h\left(b_{1}\right)}$.
So $t_{j}-t_{j+1}\le t_{j}-t_{i}=\i{h\left(c\right)}<r^{\prime\prime}$,
a contradiction. Therefore $t_{i}+\i{h\left(b_{1}\right)}+\i{h\left(c\right)}\notin\i{h\left(E_{r}\right)}$.

This shows that $\i{h\left(E_{r}\right)}\subseteq\left(-r^{\prime},r^{\prime}+r^{\prime\prime}\right)\subseteq\left(-\frac{r}{2},\frac{r}{2}\right)$.
\end{proof}
\begin{notation}
For $D\subseteq G$ and $b\in G$, denote $\Delta_{D,b}:=\left(D\bigtriangleup\left(D+b\right)\right)\cup\left(D\bigtriangleup\left(D-b\right)\right)$.
\end{notation}

\begin{lem}
\label{almost_invariance_under_certain_infinitesimal_translations}Let
$D$ be a $\Z$-definable subset of $G$, and let $\left(\epsilon_{i}\right)_{i<\omega}$
be indiscernible over $\Z$. Denote $\epsilon:=\epsilon_{1}-\epsilon_{0}$.
Then $h\left(\Delta_{D,\epsilon}\right)$ is finite.
\end{lem}

\begin{proof}
Suppose otherwise. Without loss of generality, $h\left(D\bigtriangleup\left(D+\epsilon\right)\right)$
is infinite. Denote $B_{i}:=\left(D+\epsilon_{2i}\right)\bigtriangleup\left(D+\epsilon_{2i+1}\right)$.
So $h\left(B_{0}\right)$ is infinite. Since the theory is NIP, the
family $\left\{ B_{i}\right\} _{i<\omega}$ is $k$-inconsistent for
some $k<\omega$.

Note that $h\left(B_{i}\right)$ does not depend on $i$: Let $i,j<\omega$,
and let $u\in h\left(B_{j}\right)$. So $B_{j}\cap h^{-1}\left(u\right)\neq\emptyset$.
By indiscernibility, there is $\sigma\in\text{Aut}\left(G\big/\Z\right)$
such that $\sigma\left(\epsilon_{2j}\right)=\epsilon_{2i}$ and $\sigma\left(\epsilon_{2j+1}\right)=\epsilon_{2i+1}$.
So $\sigma\left(B_{j}\right)=B_{i}$, therefore, since $h^{-1}\left(u\right)$
is type-definable over $\Z$, $B_{i}\cap h^{-1}\left(u\right)=\sigma\left(B_{j}\cap h^{-1}\left(u\right)\right)\neq\emptyset$,
which gives $u\in h\left(B_{i}\right)$.

Let $\left\{ u_{n}\,:\,n<\omega\right\} \subseteq\i{h\left(B_{0}\right)}\backslash\left\{ -\frac{1}{2}\right\} $
be distinct, and let $c_{n}\in G$ be such that $\i{h\left(c_{n}\right)}=u_{n}$.
Also fix a non-principal ultrafilter $U$ on $\omega$.

Fix $N<\omega$. Then there is $0<r_{N}<\frac{1}{2}$ such that for
every $0\le n,m\le N-1$, $n\neq m$, we have $\left(u_{n}-2r_{N},u_{n}+2r_{N}\right)\subseteq\left(-\frac{1}{2},\frac{1}{2}\right)$
and $\left(u_{n}-2r_{N},u_{n}+2r_{N}\right)\cap\left(u_{m}-2r_{N},u_{m}+2r_{N}\right)=\emptyset$.
By \cref{local_genericity} there is a finite $F\subseteq\Z$ such
that 
\[
h^{-1}\left(\left[-r_{N},r_{N}\right]\right)\subseteq E_{r_{N}}+F\subseteq h^{-1}\left(\left(-2r_{N},2r_{N}\right)\right)
\]
(where $E_{r_{N}}$ is given by \cref{uniformly_def_family_of_approximate_balls}).
Let $0\le n\le N-1$. Then 
\[
h^{-1}\left(\left[u_{n}-r_{N},u_{n}+r_{N}\right]\right)\subseteq E_{r_{N}}+F+c_{n}\subseteq h^{-1}\left(\left(u_{n}-2r_{N},u_{n}+2r_{N}\right)\right)
\]

For each $i<\omega$, $u_{n}\in\i{h\left(B_{0}\right)}=\i{h\left(B_{i}\right)}$,
so $B_{i}\cap h^{-1}\left(q\left(u_{n}\right)\right)\neq\emptyset$,
and therefore $B_{i}\cap\left(E_{r_{N}}+F+c_{n}\right)\neq\emptyset$.
So for some $d\in F$ we have $B_{i}\cap\left(E_{r_{N}}+d+c_{n}\right)\neq\emptyset$.
For each $d\in F$ let $I_{N,n,d}:=\left\{ i<\omega\,:\,B_{i}\cap\left(E_{r_{N}}+d+c_{n}\right)\neq\emptyset\right\} $.
Then $\bigcup_{d\in F}I_{N,n,d}=\omega$, so there is $d_{N,n}\in F$
such that $I_{N,n,d_{N,n}}\in U$. Let $I_{N}:=\bigcap_{n=0}^{N-1}I_{N,n,d_{N,n}}$.
So $I_{N}\in U$, and in particular is infinite. Denote $e_{N,n}:=d_{N,n}+c_{n}$. 

Note that $E_{r_{N}}+e_{N,n}\subseteq h^{-1}\left(\left(u_{n}-2r_{N},u_{n}+2r_{N}\right)\right)$,
so for every $0\le n,m\le N-1$, $n\neq m$, $\left(E_{r_{N}}+e_{N,n}\right)\cap\left(E_{r_{N}}+e_{N,m}\right)\subseteq h^{-1}\left(\left(u_{n}-2r_{N},u_{n}+2r_{N}\right)\cap\left(u_{m}-2r_{N},u_{m}+2r_{N}\right)\right)=\emptyset$,
i.e., the family $\left\{ E_{r_{N}}+e_{N,n}\,:\,0\le n\le N-1\right\} $
is $2$-inconsistent.

By the definition of $I_{N}$, for every $i\in I_{N}$ and every $0\le n\le N-1$
we have $B_{i}\cap\left(E_{r_{N}}+e_{N,n}\right)\neq\emptyset$. Note
that the family $\left\{ E_{r_{N}}+e_{N,n}\,:\,N<\omega,\,0\le n\le N-1\right\} $
is uniformly definable. Therefore, and since $I_{N}$ is infinite
and $N$ is arbitrary, by compactness we get an inp-pattern of depth
$2$, contradicting the dp-minimality of $\mathcal{Z}$.
\end{proof}
\begin{cor}
\label{only_finitely_many_isolated_points}Let $D$ be a $\Z$-definable
subset of $G$. Then $h\left(D\right)$ has only finitely many isolated
points. 
\end{cor}

\begin{proof}
Let $\left(\epsilon_{i}\right)_{i<\omega}$ be indiscernible over
$\Z$, and denote $\epsilon:=\epsilon_{1}-\epsilon_{0}$. By indiscernibility,
$h\left(\epsilon_{i}\right)$ does not depend on $i$, so $h\left(\epsilon\right)=0$.
Let $s$ be an isolated point of $h\left(D\right)$. By \cref{isolated_pt_of_img_of_Z_def_set_comes_frm_sngle_elmnt_of_Z},
there is $b\in\Z$ such that $D\cap h^{-1}\left(s\right)=\left\{ b\right\} $.
So $h\left(b+\epsilon\right)=h\left(b\right)+h\left(\epsilon\right)=s$,
hence $b+\epsilon\notin D$, therefore $b\in\Delta_{D,\epsilon}$,
so $s=h\left(b\right)\in h\left(\Delta_{D,\epsilon}\right)$. By \cref{almost_invariance_under_certain_infinitesimal_translations},
$h\left(\Delta_{D,\epsilon}\right)$ is finite, so $h\left(D\right)$
has only finitely many isolated points. 
\end{proof}
\begin{notation}
For a finite set $F\subseteq\R$ denote $B_{F,r}:=\bigcup_{u\in F}B_{u,r}$.
This set is $\Z$-definable.
\end{notation}

\begin{rem}
\label{intersection_of_arbitrarily_small_nbhds_of_a_finite_set}For
every $u\in\R$, $\bigcap_{r>0}B_{u,r}=h^{-1}\left(q\left(u\right)\right)$.
More generally, for every finite $F\subseteq\R$, $\bigcap_{r>0}B_{F,r}=h^{-1}\left(q\left(F\right)\right)$. 
\end{rem}

\begin{cor}
\label{almost_invariance_under_certain_small_translations}Let
$D$ and $\epsilon$ be as in \cref{almost_invariance_under_certain_infinitesimal_translations},
and denote $F:=h\left(\Delta_{D,\epsilon}\right)$. Then for every
$0<r\in\R$ there is $d_{r}\in\Z$ such that:
\begin{enumerate}
\item $0<\i{h\left(d_{r}\right)}<r$, and
\item $h\left(\Delta_{D,d_{r}}\right)\subseteq F+q\left(\left(-r,r\right)\right)$.
\end{enumerate}
So every $a\in G$ with $h\left(a\right)\notin F+q\left(\left(-r,r\right)\right)$
satisfies $a\in D\iff a-d_{r}\in D$ and $a\in D\iff a+d_{r}\in D$. 
\end{cor}

\begin{proof}
By \cref{almost_invariance_under_certain_infinitesimal_translations},
$F$ is finite. Let $\cE:=\left\{ b\in G\,:\,\Delta_{D,b}\subseteq h^{-1}\left(F\right)\right\} $,
and for $0<r\in\R$ let $\cE_{r}:=\left\{ b\in G\,:\,\Delta_{D,b}\subseteq B_{\i F,r}\right\} $.
Note that $\left\{ \left(a,b\right)\in G^{2}\,:\,a\in\Delta_{D,b}\right\} $
is $\Z$-definable, so $\cE_{r}$ is $\Z$-definable. By \cref{intersection_of_arbitrarily_small_nbhds_of_a_finite_set},
$\bigcap_{r>0}\cE_{r}=\cE$ and $\bigcap_{r>0}B_{0,r}=h^{-1}\left(0\right)$. 

Recall that $\epsilon=\epsilon_{1}-\epsilon_{0}$. By indiscernibility,
$h\left(\epsilon_{i}\right)$ does not depend on $i$, so $h\left(\epsilon\right)=0$.
By the definition of $F$, $\epsilon\in\cE$. So for every $0<r<\frac{1}{2}$,
$\epsilon\in\cE_{r}\cap B_{0,r}$. By elementarity, and since $\epsilon\neq0$,
there is $0\neq d_{r}\in\Z$ such that $d_{r}\in\cE_{r}\cap B_{0,r}$.

By the definition of $B_{0,r}$, and since $0<r<\frac{1}{2}$, $\i{h\left(d_{r}\right)}\in\left(-r,r\right)$.
Since $0\neq d_{r}\in\Z$, and since $\iota$ and $h\restriction_{\Z}$
are injective, $\i{h\left(d_{r}\right)}\neq0$. Note that $\Delta_{D,b}=\Delta_{D,-b}$,
so $\cE_{r}=-\cE_{r}$. Moreover, by assumption, $B_{0,r}=-B_{0,r}$.
So $-d_{r}\in\cE_{r}\cap B_{0,r}$. Since $0<r<\frac{1}{2}$, $\i{h\left(-d_{r}\right)}=-\i{h\left(d_{r}\right)}$.
Therefore, by replacing $d_{r}$ with $-d_{r}$ if necessary, we may
assume that $\i{h\left(d_{r}\right)}>0$, so $0<\i{h\left(d_{r}\right)}<r$.

By the definition of $\cE_{r}$, $\Delta_{D,d_{r}}\subseteq B_{\i F,r}\subseteq h^{-1}\left(\bigcup_{u\in\i F}q\left(\left(u-r,u+r\right)\right)\right)$,
so $h\left(\Delta_{D,d_{r}}\right)\subseteq\bigcup_{u\in\i F}q\left(\left(u-r,u+r\right)\right)$.
Since $q\left(\left(u-r,u+r\right)\right)=q\left(\left(-r,r\right)+u\right)=q\left(\left(-r,r\right)\right)+q\left(u\right)$,
we get $h\left(\Delta_{D,d_{r}}\right)\subseteq q\left(\i F\right)+q\left(\left(-r,r\right)\right)=F+q\left(\left(-r,r\right)\right)$.

For $r\ge\frac{1}{2}$, let $d_{r}:=d_{\frac{1}{4}}$.
\end{proof}
Let $D^{\prime}$ be a $\Z$-definable subset of $G$ such that $h^{-1}\left(\left[-\frac{1}{16},\frac{1}{16}\right]\right)\subseteq D^{\prime}\subseteq h^{-1}\left(\left(-\frac{1}{8},\frac{1}{8}\right)\right)$
(e.g., $D^{\prime}:=B_{0,\frac{1}{8}}$). By \cref{only_finitely_many_isolated_points}
$h\left(D^{\prime}\right)$ has only finitely many isolated points.
By \cref{isolated_pt_of_img_of_Z_def_set_comes_frm_sngle_elmnt_of_Z},
for each isolated point $s$ of $h\left(D^{\prime}\right)$ there
is $b\in\Z$ such that $D^{\prime}\cap h^{-1}\left(s\right)=\left\{ b\right\} $.
By throwing away all these $b$'s we may assume that $h\left(D^{\prime}\right)$
has no isolated points.

Let $\left(\epsilon_{i}\right)_{i<\omega}$ be indiscernible over
$\Z$, and denote $\epsilon:=\epsilon_{1}-\epsilon_{0}$. By indiscernibility,
$h\left(\epsilon_{i}\right)$ does not depend on $i$, so $h\left(\epsilon\right)=0$.
By \cref{almost_invariance_under_certain_infinitesimal_translations},
$F^{\prime}:=h\left(\Delta_{D^{\prime},\epsilon}\right)$ is finite.
We note that $F^{\prime}\subseteq h\left(D^{\prime}\right)$: if $b\in\Delta_{D^{\prime},\epsilon}$,
then at least one of $b,b+\epsilon,b-\epsilon$ is in $D^{\prime}$,
but $h\left(b+\epsilon\right)=h\left(b-\epsilon\right)=h\left(b\right)$,
so $h\left(b\right)\in D^{\prime}$. So $\i{F^{\prime}}\subseteq\i{h\left(D^{\prime}\right)}\subseteq\left(-\frac{1}{8},\frac{1}{8}\right)$.
Denote $\i{F^{\prime}}=\left\{ w_{0},\dots,w_{L^{\prime}}\right\} $
such that $-\frac{1}{8}<w_{0}<\dots<w_{L^{\prime}}<\frac{1}{8}$.

Let $\left\{ d_{r}^{\prime}\,:\,0<r\in\R\right\} $ be as in \cref{almost_invariance_under_certain_small_translations}
for $D^{\prime}$, $\epsilon$. By \cref{if_P_is_type_def_then_its_img_under_h_is_closed},
$h\left(D^{\prime}\right)$ is closed, so $C:=\i{h\left(D^{\prime}\right)}$
is closed as well. Let $s^{\prime}:=\min\left(C\right)$ and $t^{\prime}:=\max\left(C\right)$.
\begin{prop}
$s^{\prime}=w_{0}$, $t^{\prime}=w_{L^{\prime}}$, and $s^{\prime}\neq t^{\prime}$,
so $L^{\prime}\ge1$.
\end{prop}

\begin{proof}
Since $\i{F^{\prime}}\subseteq C$, we have $s^{\prime}\le w_{0}\le w_{L^{\prime}}\leq t^{\prime}$.
Suppose towards a contradiction that $s^{\prime}<w_{0}$, and let
$0<r<\frac{1}{4}$ be such that $s^{\prime}<w_{0}-r$. Since $s^{\prime}\in\i{h\left(D^{\prime}\right)}$,
there is $b\in D^{\prime}$ such that $s^{\prime}=\i{h\left(b\right)}$.
So $\i{h\left(b\right)}\notin\i{F^{\prime}}+\left(-r,r\right)$. Since
$F^{\prime}\subseteq q\left(\left(-\frac{1}{4},\frac{1}{4}\right)\right)$
and $0<r<\frac{1}{4}$, we get $h\left(b\right)\notin F^{\prime}+q\left(\left(-r,r\right)\right)$.
By the choice of $d_{r}^{\prime}$, $b\in D^{\prime}\iff b-d_{r}^{\prime}\in D^{\prime}$,
therefore $b-d_{r}^{\prime}\in D^{\prime}$. Since $h\left(b\right),h\left(d_{r}^{\prime}\right)\in q\left(\left(-\frac{1}{4},\frac{1}{4}\right)\right)$,
we get $s^{\prime}-\i{h\left(d_{r}^{\prime}\right)}=\i{h\left(b-d_{r}^{\prime}\right)}\in\i{h\left(D^{\prime}\right)}$.
Since $\i{h\left(d_{r}^{\prime}\right)}>0$, we get $s^{\prime}-\i{h\left(d_{r}^{\prime}\right)}<s^{\prime}$,
a contradiction to the definition of $s^{\prime}$. So $s^{\prime}=w_{0}$,
and the proof that $t^{\prime}=w_{L^{\prime}}$ is analogous. Finally,
if $s^{\prime}=t^{\prime}$ then $h\left(D^{\prime}\right)=\left\{ q\left(s^{\prime}\right)\right\} $,
contradicting the fact that $h\left(D^{\prime}\right)$ has no isolated
points. So $s^{\prime}\neq t^{\prime}$, and therefore $L^{\prime}\ge1$.
\end{proof}
For $0\le i\le L^{\prime}-1$, let $I_{i}:=\left(w_{i},w_{i+1}\right)$
and $\tilde{I}_{i}:=h^{-1}\left(q\left(I_{i}\right)\right)$. If for
all $0\le i\le L^{\prime}-1$ we have $D^{\prime}\cap\tilde{I}_{i}\neq\emptyset$,
let $D:=D^{\prime}$, $F:=F^{\prime}$, $L:=L^{\prime}$, $s:=s^{\prime}$,
$t:=t^{\prime}$. Otherwise, let $L:=\min\left\{ 0\le i\le L^{\prime}-1\,:\,D^{\prime}\cap\tilde{I}_{i}=\emptyset\right\} $.
If $L=0$, then $q\left(w_{0}\right)=q\left(s^{\prime}\right)$ is
an isolated point of $h\left(D^{\prime}\right)$, a contradiction.
So $L\ge1$. Let $B:=B_{-\frac{1}{8},s^{\prime},w_{L},w_{L+1}}$,
and let $D:=D^{\prime}\cap B$. So $D$ is $\Z$-definable. Let $F:=h\left(\Delta_{D,\epsilon}\right)$,
$s:=\min\left(\i{h\left(D\right)}\right)$ and $t:=\max\left(\i{h\left(D\right)}\right)$.
\begin{prop}
\label{D_intersects_every_I_i}We have $F\subseteq h\left(D\right)$,
$s=w_{0}$, $t=w_{L}$, and $\i F=\left\{ w_{0},\dots,w_{L}\right\} $.
Moreover, for all $0\le i\le L-1$ we have $D\cap\tilde{I}_{i}\neq\emptyset$.
\end{prop}

\begin{proof}
If for all $0\le i\le L^{\prime}-1$ we had $D^{\prime}\cap\tilde{I}_{i}\neq\emptyset$,
this is clear, so suppose otherwise. The proof that $F\subseteq h\left(D\right)$
is the same as for $F^{\prime}\subseteq h\left(D^{\prime}\right)$.
Denote $J:=\left[s^{\prime},w_{L}\right]$ and $\tilde{J}:=h^{-1}\left(q\left(J\right)\right)$.
By the definition of $B$, $\tilde{J}\subseteq B\subseteq h^{-1}\left(\left(-\frac{1}{8},w_{L+1}\right)\right)$,
so by the choice of $L$ and the definition of $s^{\prime}$ we have
$D=D^{\prime}\cap\tilde{J}$. In particular, we have $s=w_{0}$, $t=w_{L}$.

Since $h\left(\epsilon\right)=0$, we have $\tilde{J}+\epsilon=\tilde{J}$,
so $D+\epsilon=\left(D^{\prime}+\epsilon\right)\cap\left(\tilde{J}+\epsilon\right)=\left(D^{\prime}+\epsilon\right)\cap\tilde{J}$.
Therefore $D\backslash\left(D+\epsilon\right)=\left(D^{\prime}\cap\tilde{J}\right)\backslash\left(\left(D^{\prime}+\epsilon\right)\cap\tilde{J}\right)=\left(D^{\prime}\backslash\left(D^{\prime}+\epsilon\right)\right)\cap\tilde{J}$,
and similarly, $\left(D+\epsilon\right)\backslash D=\left(\left(D^{\prime}+\epsilon\right)\backslash D^{\prime}\right)\cap\tilde{J}$.
So $D\Delta\left(D+\epsilon\right)=\left(D^{\prime}\Delta\left(D^{\prime}+\epsilon\right)\right)\cap\tilde{J}$,
and similarly we get $D\Delta\left(D-\epsilon\right)=\left(D^{\prime}\Delta\left(D^{\prime}-\epsilon\right)\right)\cap\tilde{J}$.
Therefore $\Delta_{D,\epsilon}=\Delta_{D^{\prime},\epsilon}\cap\tilde{J}$.
By the definition of $\tilde{J}$ we get $\i F=\i{h\left(\Delta_{D,\epsilon}\right)}=\i{h\left(\Delta_{D^{\prime},\epsilon}\right)\cap q\left(J\right)}=\i{h\left(\Delta_{D^{\prime},\epsilon}\right)}\cap J=\i{F^{\prime}}\cap J=\left\{ w_{0},\dots,w_{L}\right\} $.

Finally, for each $0\le i\le L-1$ we have $\tilde{I}_{i}\subseteq\tilde{J}$,
so $D\cap\tilde{I}_{i}=D^{\prime}\cap\tilde{J}\cap\tilde{I}_{i}=D^{\prime}\cap\tilde{I}_{i}\neq\emptyset$.
\end{proof}
Let $\left\{ d_{r}\,:\,0<r\in\R\right\} $ be as in \cref{almost_invariance_under_certain_small_translations}
for $D$, $\epsilon$.
\begin{prop}
\label{every_element_in_img_of_D_cap_I_is_an_accumulation_pt}For
all $0\le i\le L-1$ and $b\in D\cap\tilde{I}_{i}$, $\i{h\left(b\right)}$
is an accumulation point of $\i{h\left(D\right)}$. 
\end{prop}

\begin{proof}
Let $0<r<\frac{1}{4}$. Let $0<\rho<r$ be such that $\i{h\left(b\right)}\in\left(w_{i}+\rho,w_{i+1}-\rho\right)$.
So $\i{h\left(b\right)}\notin\i F+\left(-\rho,\rho\right)$, hence
$h\left(b\right)\notin F+q\left(\left(-\rho,\rho\right)\right)$,
and therefore $b\in D\iff b+d_{\rho}\in D$. So $b+d_{\rho}\in D$,
hence $\i{h\left(b+d_{\rho}\right)}\in\i{h\left(D\right)}$. Since
$\i{h\left(d_{\rho}\right)}\in\left(0,\rho\right)$, we get $\i{h\left(b+d_{\rho}\right)}=\i{h\left(b\right)}+\i{h\left(d_{\rho}\right)}\in\left(\i{h\left(b\right)},\i{h\left(b\right)}+\rho\right)\subseteq\left(\i{h\left(b\right)},\i{h\left(b\right)}+r\right)$.
\end{proof}
For each $0\le i\le L-1$, $\tilde{I}_{i}$ is $\bigvee$-definable
over $\Z$, so $D\cap\tilde{I}_{i}$ is $\bigvee$-definable over
$\Z$. By \cref{D_intersects_every_I_i}, $D\cap\tilde{I}_{i}\neq\emptyset$,
so $D\cap\tilde{I}_{i}\cap\Z\neq\emptyset$ as well. Fix $b_{i}\in D\cap\tilde{I}_{i}\cap\Z$,
and let $u_{i}:=\i{h\left(b_{i}\right)}$. Let $0<\rho\in\R$ be such
that for all $0\le i\le L-1$, $u_{i}\in\left(w_{i}+4\rho,w_{i+1}-4\rho\right)$,
and denote $J_{i}:=\left(u_{i}-\rho,u_{i}+\rho\right)$, $\tilde{J}_{i}:=h^{-1}\left(q\left(J_{i}\right)\right)$.
\begin{prop}
For each $0\le i\le L-1$ there is a finite set $\Gamma_{i}\subseteq h^{-1}\left(0\right)$
such that $\tilde{J}_{i}\subseteq D+\Gamma_{i}$.
\end{prop}

\begin{proof}
Denote $J_{i}^{\prime}:=\left(u_{i}-\frac{\rho}{2},u_{i}+\frac{\rho}{2}\right)$,
$\tilde{J}_{i}^{\prime}:=h^{-1}\left(q\left(J_{i}^{\prime}\right)\right)$,
$B_{i}:=B_{u_{i},\rho}$. So $\tilde{J}_{i}^{\prime}\subseteq B_{i}\subseteq\tilde{J}_{i}$.
By \cref{every_element_in_img_of_D_cap_I_is_an_accumulation_pt},
$u_{i}$ is an accumulation point of $\i{h\left(D\right)}$, so $\i{h\left(D\cap\tilde{J}_{i}^{\prime}\right)}=\i{h\left(D\right)}\cap J_{i}^{\prime}$
is infinite. Therefore $D\cap\tilde{J}_{i}^{\prime}$ is infinite,
and since $D\cap\tilde{J}_{i}^{\prime}\subseteq D\cap B_{i}$, $D\cap B_{i}$
is infinite as well. Since $D$ and $B_{i}$ are $\Z$-definable and
$b_{i}\in\Z$, $D\cap B_{i}-b_{i}$ is $\Z$-definable. Note that
$h\left(D\cap B_{i}-b_{i}\right)=h\left(D\cap B_{i}\right)-h\left(b_{i}\right)\subseteq h\left(\tilde{J}_{i}\right)-q\left(u_{i}\right)\subseteq q\left(J_{i}\right)-q\left(u_{i}\right)=q\left(\left(-\rho,\rho\right)\right)$,
so by \cref{local_genericity} (with $r:=2\rho$ and $u:=\rho$)
there is a finite set $C_{i}\subseteq\Z$ such that $h\left(C_{i}\right)\subseteq\left[-2\rho,2\rho\right]$
and 
\[
h^{-1}\left(\left[-\rho,\rho\right]\right)\subseteq D\cap B_{i}-b_{i}+C_{i}\subseteq h^{-1}\left(\left(-3\rho,3\rho\right)\right).
\]
Therefore, 
\[
\tilde{J}_{i}\subseteq h^{-1}\left(\left[u_{i}-\rho,u_{i}+\rho\right]\right)\subseteq D\cap B_{i}+C_{i}\subseteq h^{-1}\left(\left(u_{i}-3\rho,u_{i}+3\rho\right)\right).
\]
For each $c\in C_{i}$ let $E_{i,c}:=\left(D\cap B_{i}+c\right)\cap\tilde{J}_{i}$.
So $\tilde{J}_{i}\subseteq\bigcup_{c\in C_{i}}E_{i,c}$.
\begin{claim}
For every $c\in C_{i}$ there is $\gamma_{c}\in h^{-1}\left(0\right)$
such that $E_{i,c}\subseteq D+\gamma_{c}$. 
\end{claim}

\begin{proof}
Let $c\in C_{i}$, and let $v:=\i{h\left(c\right)}$. So $v\in\left[-2\rho,2\rho\right]$.
If $v=0$ then $\gamma_{c}:=c$ is as required, so suppose $v\neq0$.
Without loss of generality, $v>0$. 

Since $\tilde{J}_{i}$ is $\bigvee$-definable over $\Z$ and $c\in\Z$,
also $E_{i,c}$ is $\bigvee$-definable over $\Z$. Write $E_{i,c}=\bigcup_{\alpha}E_{i,c,\alpha}$,
where each $E_{i,c,\alpha}$ is $\Z$-definable. Let $P_{i,c}:=\left\{ c^{\prime}\in G\,:\,E_{i,c}\subseteq D+c^{\prime}\right\} $.
Then $P_{i,c}=\bigcap_{\alpha}\left\{ c^{\prime}\in G\,:\,E_{i,c,\alpha}\subseteq D+c^{\prime}\right\} $,
so $P_{i,c}$ is type-definable over $\Z$. 

Let $P_{i,c}^{\prime}:=P_{i,c}\cap h^{-1}\left(\left[0,v\right]\right)$.
So $P_{i,c}^{\prime}$ is type-definable over $\Z$, therefore, by
\cref{if_P_is_type_def_then_its_img_under_h_is_closed}, $\i{h\left(P_{i,c}^{\prime}\right)}$
is closed. Note that $c\in P_{i,c}$ and $\i{h\left(P_{i,c}^{\prime}\right)}=\i{h\left(P_{i,c}\right)}\cap\left[0,v\right]$,
so $v=\i{h\left(c\right)}\in\i{h\left(P_{i,c}^{\prime}\right)}\neq\emptyset$.
Let $\tilde{v}:=\min\left(\i{h\left(P_{i,c}^{\prime}\right)}\right)$,
and let $\tilde{c}\in P_{i,c}^{\prime}$ be such that $\i{h\left(\tilde{c}\right)}=\tilde{v}$.
So $0\le\tilde{v}\le v$ and $E_{i,c}\subseteq D+\tilde{c}$.

Suppose towards a contradiction that $\tilde{v}>0$. Let $0<r<\min\left(\tilde{v},\rho\right)$
and denote $c^{\prime}:=\tilde{c}-d_{r}$, $v^{\prime}:=\i{h\left(c^{\prime}\right)}=\tilde{v}-\i{h\left(d_{r}\right)}$.
Since $0<\i{h\left(d_{r}\right)}<r$ we have $0<v^{\prime}<\tilde{v}$. 

We show that $E_{i,c}\subseteq D+c^{\prime}$. Let $a\in E_{i,c}$.
In particular, $a\in\tilde{J}_{i}$, so $\i{h\left(a-\tilde{c}\right)}=\i{h\left(a\right)}-\tilde{v}\in J_{i}-\tilde{v}=\left(u_{i}-\tilde{v}-\rho,u_{i}-\tilde{v}+\rho\right)$.
Since $u_{i}\in\left(w_{i}+4\rho,w_{i+1}-4\rho\right)$ and $0\le\tilde{v}\le v\le2\rho$,
we get $\i{h\left(a-\tilde{c}\right)}\in\left(w_{i}+\rho,w_{i+1}-\rho\right)\subseteq\left(w_{i}+r,w_{i+1}-r\right)$,
so $h\left(a-\tilde{c}\right)\notin F+q\left(\left(-r,r\right)\right)$.
By the choice of $d_{r}$, $a-\tilde{c}\in D\iff a-\tilde{c}+d_{r}\in D$.
Since $E_{i,c}\subseteq D+\tilde{c}$, we have $a-\tilde{c}\in D$,
so $a-\tilde{c}+d_{r}\in D$, and therefore $a\in D+\tilde{c}-d_{r}=D+c^{\prime}$.

So $c^{\prime}\in P_{i,c}^{\prime}$, hence $v^{\prime}=\i{h\left(c^{\prime}\right)}\in\i{h\left(P_{i,c}^{\prime}\right)}$,
contradicting the minimality of $\tilde{v}$. Therefore $\tilde{v}=0$,
so $\gamma_{c}:=\tilde{c}$ is as required.
\end{proof}
Let $\Gamma_{i}:=\left\{ \gamma_{c}\,:\,c\in C_{i}\right\} $. So
$\Gamma_{i}\subseteq h^{-1}\left(0\right)$ and we have
\[
\tilde{J}_{i}\subseteq\bigcup_{c\in C_{i}}E_{i,c}\subseteq\bigcup_{c\in C_{i}}\left(D+\gamma_{c}\right)=D+\Gamma_{i}
\]
\end{proof}
\begin{prop}
$\tilde{I}_{i}\subseteq D+\Gamma_{i}$.
\end{prop}

\begin{proof}
Let
\begin{align*}
w_{i}^{\prime}:= & \min\left\{ w_{i}\le v\le u_{i}\,:\,h^{-1}\left(\left(v,u_{i}\right]\right)\subseteq D+\Gamma_{i}\right\} \\
w_{i+1}^{\prime}:= & \max\left\{ u_{i}\le v\le w_{i+1}\,:\,h^{-1}\left(\left[u_{i},v\right)\right)\subseteq D+\Gamma_{i}\right\} 
\end{align*}
So $w_{i}^{\prime}\ge w_{i}$ and $w_{i+1}^{\prime}\le w_{i+1}$,
and since $\tilde{J}_{i}\subseteq D+\Gamma_{i}$ we have $w_{i}^{\prime}\le u_{i}-\rho$
and $w_{i+1}^{\prime}\ge u_{i}+\rho$. We claim that $w_{i}^{\prime}=w_{i}$
and $w_{i+1}^{\prime}=w_{i+1}$. We show that $w_{i}^{\prime}=w_{i}$,
and the proof that $w_{i+1}^{\prime}=w_{i+1}$ is analogous. 

Suppose towards a contradiction that $w_{i}^{\prime}>w_{i}$, and
let $0<r<\min\left(\rho,w_{i}^{\prime}-w_{i}\right)$. Since $\i{h\left(d_{r}\right)}>0$,
by the definition of $w_{i}^{\prime}$ there exists $a\in G$ such
that $w_{i}^{\prime}-\i{h\left(d_{r}\right)}<\i{h\left(a\right)}\le w_{i}^{\prime}$
and $a\notin D+\Gamma_{i}$. Since $\i{h\left(d_{r}\right)}<r$ we
get 
\begin{equation}
w_{i}^{\prime}<\i{h\left(a\right)}+\i{h\left(d_{r}\right)}=\i{h\left(a+d_{r}\right)}\le w_{i}^{\prime}+\i{h\left(d_{r}\right)}<w_{i}^{\prime}+r<w_{i}^{\prime}+\rho\le u_{i}\tag{\ensuremath{\star}}\label{bounding_a_plus_d_r}
\end{equation}
and therefore, by the definition of $w_{i}^{\prime}$, $a+d_{r}\in D+\Gamma_{i}$,
so there is $\gamma\in\Gamma_{i}$ such that $a+d_{r}-\gamma\in D$. 

Since $r<\rho\le w_{i+1}^{\prime}-u_{i}\le w_{i+1}-u_{i}$ and $r<w_{i}^{\prime}-w_{i}$,
by (\cref{bounding_a_plus_d_r}) we get $w_{i}+r<\i{h\left(a+d_{r}\right)}<w_{i+1}-r$.
Since $\Gamma_{i}\subseteq h^{-1}\left(0\right)$, $\i{h\left(a+d_{r}-\gamma\right)}=\i{h\left(a+d_{r}\right)}-\i{h\left(\gamma\right)}=\i{h\left(a+d_{r}\right)}$,
so $w_{i}+r<\i{h\left(a+d_{r}-\gamma\right)}<w_{i+1}-r$, and hence
$h\left(a+d_{r}-\gamma\right)\notin F+q\left(\left(-r,r\right)\right)$.
By the choice of $d_{r}$ we get $a+d_{r}-\gamma\in D\iff a-\gamma\in D$,
and therefore $a-\gamma\in D$. So $a\in D+\gamma\subseteq D+\Gamma_{i}$,
a contradiction.

So $w_{i}^{\prime}=w_{i}$, and similarly $w_{i+1}^{\prime}=w_{i+1}$,
therefore $\tilde{I}_{i}=h^{-1}\left(\left(w_{i},w_{i+1}\right)\right)\subseteq D+\Gamma_{i}$. 
\end{proof}
\begin{prop}
\label{externally_definable_short_interval_starting_at_0}There
exists $0<u<\frac{1}{4}$ such that $h^{-1}\left(\left(0,u\right)\right)\cap\Z$
is externally definable in $\mathcal{Z}$.
\end{prop}

\begin{proof}
Let $\Gamma:=\bigcup_{i=0}^{L-1}\Gamma_{i}$. So $\bigcup_{i=0}^{L-1}\tilde{I}_{i}\subseteq D+\Gamma$.
Since $D\subseteq h^{-1}\left(\left[w_{0},w_{L}\right]\right)$ and
$\Gamma\subseteq h^{-1}\left(0\right)$, also $D+\Gamma\subseteq h^{-1}\left(\left[w_{0},w_{L}\right]\right)$.
Since $\Gamma$ is finite, $D+\Gamma$ is definable.

Let $0<u<\min\left(w_{1}-w_{0},w_{L}-w_{L-1}\right)<\frac{1}{4}$.
Since $w_{L}-w_{0}<\frac{1}{4}$, we have $0<w_{L}-w_{0}-u<\frac{1}{4}$,
so there is $e\in G$ such that $\i{h\left(e\right)}=w_{L}-w_{0}-u$.
Let $B_{1}:=\left(D+\Gamma\right)\cap\left(D+\Gamma+e\right)$. So
$h^{-1}\left(\left(w_{L}-u,w_{L}\right)\right)\subseteq h^{-1}\left(\left(w_{L-1},w_{L}\right)\right)=\tilde{I}_{L-1}\subseteq D+\Gamma$,
and $h^{-1}\left(\left(w_{L}-u,w_{L}\right)\right)-e\subseteq h^{-1}\left(\left(w_{0},w_{1}\right)\right)=\tilde{I}_{0}\subseteq D+\Gamma$,
therefore $h^{-1}\left(\left(w_{L}-u,w_{L}\right)\right)\subseteq B_{1}$.
We also have $\i{h\left(B_{1}\right)}\subseteq\i{h\left(D+\Gamma\right)}\cap\i{h\left(D+\Gamma+e\right)}=\i{h\left(D+\Gamma\right)}\cap\left(\i{h\left(D+\Gamma\right)}+\i{h\left(e\right)}\right)\subseteq\left[w_{0},w_{L}\right]\cap\left[w_{L}-u,2w_{L}-w_{0}-u\right]=\left[w_{L}-u,w_{L}\right]$,
so $B_{1}\subseteq h^{-1}\left(\left[w_{L}-u,w_{L}\right]\right)$. 

Let $b\in G$ be such that $\i{h\left(b\right)}=w_{L}-u$, and let
$B_{2}:=B_{1}-b$. So $h^{-1}\left(\left(0,u\right)\right)\subseteq B_{2}\subseteq h^{-1}\left(\left[0,u\right]\right)$.
If $u\in\i{h\left(\Z\right)}$, then, since $\iota$ and $h\restriction_{\Z}$
are injective, there is a unique $c\in\Z$ such that $u=\i{h\left(c\right)}$.
Otherwise, let $c=0$. Let $B_{3}:=B_{2}\backslash\left\{ 0,c\right\} $.
So $B_{3}$ is definable and satisfies $h^{-1}\left(\left(0,u\right)\right)\subseteq B_{3}\subseteq h^{-1}\left(\left[0,u\right]\right)\backslash\left\{ 0,c\right\} $.
Therefore $B_{3}\cap\Z=h^{-1}\left(\left(0,u\right)\right)\cap\Z$
is externally definable in $\mathcal{Z}$.
\end{proof}
Recall that in the beginning of this section we chose $\alpha\in\R$
such that $h\left(1\right)=\alpha+\Z=q\left(\alpha\right)$, and that
$\alpha\in\R\backslash\Q$. Denote by $\eta:\Z\rightarrow\R\big/\Z$
the function given by $\eta\left(n\right):=n\alpha+\Z$, i.e., $\eta=h\restriction_{\Z}$.
So for any $S\subseteq\R\big/\Z$ we have $\eta^{-1}\left(S\right)=h^{-1}\left(S\right)\cap\Z$.
\begin{lem}
\label{defining_the_cyclic_order_from_a_short_interval_starting_at_0}Let
$\cM$ be an expansion of $\left(\mathbb{Z},+,0,1\right)$, and suppose
that there exists $0<u<\frac{1}{4}$ such that $J:=\eta^{-1}\left(q\left(\left(0,u\right)\right)\right)$
is definable in $\cM$. Then the cyclic order $C_{\alpha}$ is definable
in $\cM$.
\end{lem}

\begin{proof}
Define $R\left(x,y\right)$ by $y-x\in J$. Then for $a,b\in\Z$ we
have $R\left(a,b\right)$ exactly when either $\i{\eta\left(a\right)}<\i{\eta\left(b\right)}<\i{\eta\left(a\right)}+u$
or $1+\i{\eta\left(b\right)}<\i{\eta\left(a\right)}+u$ (which implies
$\i{\eta\left(b\right)}<\i{\eta\left(a\right)}$). Also note that
$R\left(a,b\right)\implies\neg R\left(b,a\right)$.

Since $\i{\eta\left(\Z\right)}$ is dense in $\left(-\frac{1}{2},\frac{1}{2}\right)$,
we can choose $e_{0},\dots,e_{N}\in\Z$ for some $N<\omega$ such
that $-\frac{1}{2}<\i{\eta\left(e_{0}\right)}<\i{\eta\left(e_{1}\right)}<\dots<\i{\eta\left(e_{N}\right)}<\frac{1}{2}$
and for all $0\le i\le N-1$, $\i{\eta\left(e_{i+1}\right)}-\i{\eta\left(e_{i}\right)}<u$,
and $1+\i{\eta\left(e_{0}\right)}-\i{\eta\left(e_{N}\right)}<u$ (so
in particular, $\frac{1}{2}-\i{\eta\left(e_{N}\right)}<u$ and $\i{\eta\left(e_{0}\right)}+\frac{1}{2}<u$).
For each $0\le i\le N-1$, let $I_{i}\subseteq\Z$ be the set defined
by $\left(R\left(e_{i},x\right)\wedge R\left(x,e_{i+1}\right)\right)\vee\left(x=e_{i}\right)$,
and let $I_{N}\subseteq\Z$ be the set defined by $\left(R\left(e_{N},x\right)\wedge R\left(x,e_{0}\right)\right)\vee\left(x=e_{N}\right)$.
So for $a\in\Z$ we have that for each $0\le i\le N-1$, $a\in I_{i}$
if and only if $\i{\eta\left(e_{i}\right)}\le\i{\eta\left(a\right)}<\i{\eta\left(e_{i+1}\right)}$,
and $a\in I_{N}$ if and only if $\i{\eta\left(e_{N}\right)}\le\i{\eta\left(a\right)}<\frac{1}{2}$
or $-\frac{1}{2}\le\i{\eta\left(a\right)}<\i{\eta\left(e_{0}\right)}$.
Therefore the sets $\left\{ I_{i}\right\} _{i=0}^{N}$ are disjoint
and $\bigcup_{i=0}^{N}I_{i}=\Z$. Let $\xi:\Z\to\left\{ 0,\dots,N\right\} $
be such that $\xi\left(a\right)$ is the unique $0\le i\le N$ for
which $a\in I_{i}$. Since the sets $\left\{ I_{i}\right\} _{i=0}^{N}$
are definable in $\cM$, the relations $\xi\left(x\right)=\xi\left(y\right)$
and $\xi\left(x\right)<\xi\left(y\right)$ are also definable in $\cM$.

Note that for $a,b,c\in\Z$ we have $C_{\alpha}\left(a,b,c\right)$
if and only if $\i{\eta\left(a\right)}<\i{\eta\left(b\right)}<\i{\eta\left(c\right)}$
or $\i{\eta\left(b\right)}<\i{\eta\left(c\right)}<\i{\eta\left(a\right)}$
or $\i{\eta\left(c\right)}<\i{\eta\left(a\right)}<\i{\eta\left(b\right)}$.
Therefore we have:
\begin{itemize}
\item If $\xi\left(a\right)=\xi\left(b\right)=\xi\left(c\right)$, then
\[
C_{\alpha}\left(a,b,c\right)\iff\left(R\left(a,b\right)\wedge R\left(b,c\right)\right)\vee\left(R\left(b,c\right)\wedge R\left(c,a\right)\right)\vee\left(R\left(c,a\right)\wedge R\left(a,b\right)\right)
\]
\item If $\xi\left(a\right)=\xi\left(b\right)\neq\xi\left(c\right)$, then
$C_{\alpha}\left(a,b,c\right)\iff R\left(a,b\right)$.
\item If $\xi\left(b\right)=\xi\left(c\right)\neq\xi\left(a\right)$, then
$C_{\alpha}\left(a,b,c\right)\iff C_{\alpha}\left(b,c,a\right)\iff R\left(b,c\right)$.
\item If $\xi\left(c\right)=\xi\left(a\right)\neq\xi\left(b\right)$, then
$C_{\alpha}\left(a,b,c\right)\iff C_{\alpha}\left(c,a,b\right)\iff R\left(c,a\right)$.
\item If $\xi\left(a\right)$,$\xi\left(b\right)$,$\xi\left(c\right)$
are pairwise distinct, then 
\begin{gather*}
C_{\alpha}\left(a,b,c\right)\iff C_{\alpha}\left(e_{\xi\left(a\right)},e_{\xi\left(b\right)},e_{\xi\left(c\right)}\right)\iff\\
\iff\xi\left(a\right)<\xi\left(b\right)<\xi\left(c\right)\text{ or }\xi\left(b\right)<\xi\left(c\right)<\xi\left(a\right)\text{ or }\xi\left(c\right)<\xi\left(a\right)<\xi\left(b\right)
\end{gather*}
\end{itemize}
Combining these, we get that $C_{\alpha}$ is definable in $\cM$.
\end{proof}
From \cref{externally_definable_short_interval_starting_at_0}
and \cref{defining_the_cyclic_order_from_a_short_interval_starting_at_0}
we get:
\begin{cor}
\label{the_cyclic_order_is_def_in_the_shelah_expansion}The cyclic
order $C_{\alpha}$ is definable in $\cZ^{Sh}$.
\end{cor}

We now work to upgrade \cref{the_cyclic_order_is_def_in_the_shelah_expansion}
and obtain that $C_{\alpha}$ is definable in $\mathcal{Z}$. The
following is a special case of \cite[Theorem 6.1]{SW2019} where $\left(M,+,C\right)$
is $\left(\mathbb{Z},+,C_{\alpha}\right)$ and $\cN=\cM$: 
\begin{fact}
\label{structure_of_unary_subsets_in_dp_minimal_expansions_of_cyclic_order}Let
$\cM$ be a dp-minimal expansion of $\left(\mathbb{Z},+,0,1,C_{\alpha}\right)$.
Then every subset of $\Z$ which is definable in $\cM$ is a finite
union of sets of the form $a+nJ$ for $a\in\Z$, $1\le n<\omega$,
and $J\subseteq\Z$ convex with respect to $C_{\alpha}$.
\end{fact}

In the above, a subset $J\subseteq\Z$ is convex with respect to $C_{\alpha}$
if for every $a,b\in J$ such that $a\neq b$, we have either $\left\{ k\in\Z\,:\,C_{\alpha}\left(a,k,b\right)\right\} \subseteq J$
or $\left\{ k\in\Z\,:\,C_{\alpha}\left(b,k,a\right)\right\} \subseteq J$.
Equivalently, $J$ is convex if and only if there is an interval $I\subseteq\R$
of length at most $1$ such that $J=\eta^{-1}\left(q\left(I\right)\right)$.

Let $A\subseteq\Z$ be infinite and definable in $\cZ$ such that
$A\subseteq\eta^{-1}\left(q\left(\left(-\frac{1}{32},\frac{1}{32}\right)\right)\right)$
(e.g., $A:=B_{0,\frac{1}{32}}\cap\Z$, see \cref{approximate_balls_notation}).
So $A$ is definable in $\cZ^{Sh}$, which, by \cref{the_cyclic_order_is_def_in_the_shelah_expansion}
and \cref{shelah_expansion_preserves_dp_rank}, is a dp-minimal
expansion of $\left(\mathbb{Z},+,0,1,C_{\alpha}\right)$. So by \cref{structure_of_unary_subsets_in_dp_minimal_expansions_of_cyclic_order},
$A=\bigcup_{i=0}^{N}\left(a_{i}+n_{i}J_{i}\right)$ for some $N<\omega$,
where for each $i\le N$, $a_{i}\in\Z$, $1\le n_{i}<\omega$, and
$J_{i}\subseteq\Z$ convex with respect to $C_{\alpha}$. Let $I_{i}\subseteq\R$
be an interval of length at most $1$ such that $J_{i}=\eta^{-1}\left(q\left(I_{i}\right)\right)$.
Note that for every $s\in\R$, $\eta^{-1}\left(q\left(s\right)\right)$
is either empty or a singleton. So by throwing away at most finitely
many points from $A$, we may assume that for each $i\le N$, $I_{i}$
is open and nonempty.
\begin{notation}
For $B\subseteq\Z$ and $1\le m<\omega$ we denote $B\big/m=\frac{1}{m}B:=\left\{ a\in\Z\,:\,ma\in B\right\} $.
\end{notation}

\begin{obs}
\label{dividing_an_inteval}Let $I=\left(s,t\right)\subseteq\R$
be an interval of length at most $1$ and let $J:=\eta^{-1}\left(q\left(I\right)\right)$.
Let $1\le k<\omega$, and for each $0\le i\le k-1$ denote $I_{i}:=\left(\frac{s}{k},\frac{t}{k}\right)+\frac{i}{k}$.
Then $J\big/k=\bigcup_{i=0}^{k-1}\eta^{-1}\left(q\left(I_{i}\right)\right)$.
The same holds if we replace all open intervals with closed or half-open
intervals.
\end{obs}

For each $i\le N$, write $a_{i}=n_{i}b_{i}+r_{i}$ such that $b_{i},r_{i}\in\Z$,
$0\le r_{i}\le n_{i}-1$. Denote $p:=\i{\eta\left(r_{0}\right)}$,
$s:=-\frac{1}{32}-p$, $t:=\frac{1}{32}-p$. So by replacing $A$
with $A-r_{0}$ we may assume that $r_{0}=0$, with $A\subseteq\eta^{-1}\left(q\left(\left(s,t\right)\right)\right)$
instead of $A\subseteq\eta^{-1}\left(q\left(\left(-\frac{1}{32},\frac{1}{32}\right)\right)\right)$.
Denote $m:=\prod_{i=0}^{N}n_{i}$ and $B:=A\big/m$. Denote also $\Lambda:=\left[\frac{s}{m}-\frac{15}{32m},\frac{t}{m}+\frac{m-1}{m}+\frac{15}{32m}\right)$,
and note that the length of $\Lambda$ is exactly $1$.
\begin{prop}
\label{writing_B_with_disjoint_open_intervals_in_lambda}$B$
can be written as $B=\bigcup_{i=0}^{N^{\prime}}\eta^{-1}\left(q\left(I_{i}^{\prime}\right)\right)$
for some $N^{\prime}<\omega$, such that:
\begin{itemize}
\item for each $i\le N^{\prime}$, $I_{i}^{\prime}\subseteq\Lambda$ is
a nonempty interval,
\item $\left\{ I_{i}^{\prime}\right\} _{i=0}^{N^{\prime}}$ are pairwise
disjoint, enumerated by their order in $\R$, and
\item for each $1\le i\le N^{\prime}$, $I_{i}^{\prime}$ is open, and either
$I_{0}^{\prime}$ is open, or $I_{0}^{\prime}$ is of the form $\left[\frac{s}{m}-\frac{15}{32m},w\right)$.
\end{itemize}
\end{prop}

\begin{proof}
For each $i\le N$, denote $m_{i}:=\frac{m}{n_{i}}=\prod_{j\neq i}n_{j}$.
Clearly, if $r_{i}\neq0$ then $\frac{1}{m}\left(a_{i}+n_{i}J_{i}\right)=\emptyset$.
Conversely, if $r_{i}=0$ then $\frac{1}{m}\left(a_{i}+n_{i}J_{i}\right)=\frac{1}{m_{i}}\left(b_{i}+J_{i}\right)\neq\emptyset$,
since $\eta\left(m_{i}\Z-b_{i}\right)$ is dense in $\R\big/\Z$.
So $mB=A\cap m\Z=\bigcup\left\{ a_{i}+n_{i}J_{i}\,:\,i\le N,\,r_{i}=0\right\} $,
and since $r_{0}=0$, this union is over a nonempty set. Therefore,
and since $\left(mB\right)\big/m=B$, by replacing $A$ with $mB$
we may assume that for each $i\le N$, $r_{i}=0$. So $B=\bigcup_{i=0}^{N}\frac{1}{m_{i}}\left(b_{i}+J_{i}\right)$.

By replacing $J_{i}$ with $b_{i}+J_{i}$ we may assume that for each
$i\le N$, $b_{i}=0$, i.e., $B=\bigcup_{i=0}^{N}\frac{1}{m_{i}}J_{i}$.
For each $i$, by \cref{dividing_an_inteval} for $J_{i}=\eta^{-1}\left(q\left(I_{i}\right)\right)$
we get that $\frac{1}{m_{i}}J_{i}=\bigcup_{j=0}^{m_{i}-1}\eta^{-1}\left(q\left(I_{i,j}\right)\right)$
for nonempty open intervals $\left\{ I_{i,j}\right\} _{j=0}^{m_{i}-1}$
in $\R$, each of length at most $1$. So $B=\bigcup_{i=0}^{N}\bigcup_{j=0}^{m_{i}-1}\eta^{-1}\left(q\left(I_{i,j}\right)\right)$,
and we rewrite this as $B=\bigcup_{i=0}^{N^{\prime\prime}}\eta^{-1}\left(q\left(I_{i}^{\prime\prime}\right)\right)$.

We may assume that for each $i$, $I_{i}^{\prime\prime}\subseteq\Lambda$:
There exists $k\in\Z$ such that $\left(I_{i}^{\prime\prime}+k\right)\cap\Lambda\neq\emptyset$,
and since $I_{i}^{\prime\prime}$ is of length at most $1$, there
is a smallest such $k$, and it satisfies $I_{i}^{\prime\prime}+k\subseteq\Lambda\cup\left(\Lambda-1\right)$.
So we replace $I_{i}^{\prime\prime}$ with the two intervals $\left(I_{i}^{\prime\prime}+k\right)\cap\Lambda$
and $\left(I_{i}^{\prime\prime}+k+1\right)\cap\Lambda$ (or just the
first one, if the second one is empty). So for each $i$, either $I_{i}^{\prime\prime}$
is open, or $I_{i}^{\prime\prime}=\left[\frac{s}{m}-\frac{15}{32m},w_{i}\right)$
for some $w_{i}\le\frac{t}{m}+\frac{m-1}{m}+\frac{15}{32m}$. By combining
intervals which intersect, we can rewrite $\bigcup_{i=0}^{N^{\prime\prime}}I_{i}^{\prime\prime}$
as a finite union $\bigcup_{i=0}^{N^{\prime}}I_{i}^{\prime}$ of pairwise
disjoint intervals which satisfy all the requirements, and $B=\bigcup_{i=0}^{N^{\prime\prime}}\eta^{-1}\left(q\left(I_{i}^{\prime\prime}\right)\right)=\eta^{-1}\left(q\left(\bigcup_{i=0}^{N^{\prime\prime}}I_{i}^{\prime\prime}\right)\right)=\eta^{-1}\left(q\left(\bigcup_{i=0}^{N^{\prime}}I_{i}^{\prime}\right)\right)=\bigcup_{i=0}^{N^{\prime}}\eta^{-1}\left(q\left(I_{i}^{\prime}\right)\right)$.
\end{proof}
\begin{prop}
\label{definable_short_interval}There is an open interval $I\subseteq\R$
of length $0<L\le\frac{1}{16}$ such that $\eta^{-1}\left(q\left(I\right)\right)$
is definable in $\cZ$.
\end{prop}

\begin{proof}
Denote $K:=\left(s,t\right)$. Recall that $B=A\big/m$ and $A\subseteq\eta^{-1}\left(q\left(K\right)\right)$,
so $B\subseteq\frac{1}{m}\eta^{-1}\left(q\left(K\right)\right)$.
By \cref{dividing_an_inteval} we have $\frac{1}{m}\eta^{-1}\left(q\left(K\right)\right)=\bigcup_{i=0}^{m-1}\eta^{-1}\left(q\left(K_{i}\right)\right)$,
where $K_{i}:=\left(\frac{s}{m},\frac{t}{m}\right)+\frac{i}{m}$.
So for each $0\le i\le m-1$, $K_{i}\subseteq\left(\frac{s}{m},\frac{t}{m}+\frac{m-1}{m}\right)\subseteq\Lambda$.
Write $B=\bigcup_{i=0}^{N^{\prime}}\eta^{-1}\left(q\left(I_{i}^{\prime}\right)\right)$
as in \cref{writing_B_with_disjoint_open_intervals_in_lambda}.
So $\eta^{-1}\left(q\left(\bigcup_{i=0}^{N^{\prime}}I_{i}^{\prime}\right)\right)=B\subseteq\bigcup_{i=0}^{m-1}\eta^{-1}\left(q\left(K_{i}\right)\right)=\eta^{-1}\left(q\left(\bigcup_{i=0}^{m-1}K_{i}\right)\right)\subseteq\eta^{-1}\left(q\left(\left(\frac{s}{m},\frac{t}{m}+\frac{m-1}{m}\right)\right)\right)$. 

Denote $S:=\bigcup_{i=0}^{N^{\prime}}I_{i}^{\prime}$ and $\bar{K}_{i}:=\left[\frac{s}{m},\frac{t}{m}\right]+\frac{i}{m}$.
Let $U:=\left(\Lambda\backslash\left\{ \frac{s}{m}-\frac{15}{32m}\right\} \right)\backslash\left(\bigcup_{i=0}^{m-1}\bar{K}_{i}\right)=\left(\frac{s}{m}-\frac{15}{32m},\frac{s}{m}\right)\cup\left(\bigcup_{i=0}^{m-2}\left(\frac{t+i}{m},\frac{s+i+1}{m}\right)\right)\cup\left(\frac{t}{m}+\frac{m-1}{m},\frac{t}{m}+\frac{m-1}{m}+\frac{15}{32m}\right)$.
Suppose towards a contradiction that $U\cap S\neq\emptyset$. Note
that either $S$ is open or $S\backslash\left\{ \frac{s}{m}-\frac{15}{32m}\right\} $
is open, hence $U\cap S$ is open. Note that $q\restriction_{\Lambda\backslash\left\{ \frac{s}{m}-\frac{15}{32m}\right\} }$
is a homeomorphism, therefore, since $\eta\left(\Z\right)$ is dense
in $\R\big/\Z$, we have $\eta^{-1}\left(q\left(U\cap S\right)\right)\neq\emptyset$.
Since $q\restriction_{\Lambda}$ is a bijection, we get $\eta^{-1}\left(q\left(U\cap S\right)\right)=\eta^{-1}\left(q\left(U\right)\right)\cap\eta^{-1}\left(q\left(S\right)\right)\subseteq\eta^{-1}\left(q\left(U\right)\right)\cap\eta^{-1}\left(q\left(\bigcup_{i=0}^{m-1}K_{i}\right)\right)=\eta^{-1}\left(q\left(U\cap\bigcup_{i=0}^{m-1}K_{i}\right)\right)=\emptyset$,
a contradiction. So $S\subseteq\left\{ \frac{s}{m}-\frac{15}{32m}\right\} \cup\bigcup_{i=0}^{m-1}\bar{K}_{i}$.
But we know that for each $1\le i\le N^{\prime}$, $I_{i}^{\prime}$
is open, and either $I_{0}^{\prime}$ is open, or $I_{0}^{\prime}$
is of the form $\left[\frac{s}{m}-\frac{15}{32m},w\right)$. So we
must have that $I_{0}^{\prime}$ is open as well and that $S\subseteq\bigcup_{i=0}^{m-1}K_{i}$.

For each $i\le N^{\prime}$ denote $I_{i}^{\prime}=\left(v_{i},w_{i}\right)$.
Since $I_{i}^{\prime}$ is an interval, there is a unique $0\le j\le m-1$
such that $I_{i}^{\prime}\subseteq K_{j}$. Denote this $j$ by $\xi\left(i\right)$.
Suppose that for some $i\le N^{\prime}-1$, $w_{i}=v_{i+1}$. Then
$\xi\left(i\right)=\xi\left(i+1\right)$ and $\left(v_{i},w_{i+1}\right)\subseteq K_{\xi\left(i\right)}$.
Note that $\eta^{-1}\left(q\left(w_{i}\right)\right)$ is either empty
or a singleton, so by adding at most one point to $B=\bigcup_{i=0}^{N^{\prime}}\eta^{-1}\left(q\left(I_{i}^{\prime}\right)\right)$,
we can replace the two intervals $I_{i}^{\prime}$, $I_{i+1}^{\prime}$
with $\left(v_{i},w_{i+1}\right)$. Repeating this, we see that by
adding at most finitely many points to $B$, we may assume that for
all $i\le N^{\prime}-1$, $w_{i}<v_{i+1}$.

If $N^{\prime}=0$ then $B=\eta^{-1}\left(q\left(I_{0}^{\prime}\right)\right)$
is definable in $\cZ$, so suppose $N^{\prime}\ge1$. Let $0<r\in\R$
be such that $\frac{s}{m}-\frac{15}{32m}<v_{0}-r<v_{0}<w_{0}<w_{0}+r<v_{1}<\frac{t}{m}+\frac{m-1}{m}+\frac{15}{32m}$,
and let $B^{\prime}:=\Z\cap B_{v_{0}-r,v_{0},w_{0},w_{0}+r}$ (see
\cref{approximate_balls_notation}). So $B^{\prime}$ is definable
in $\cZ$ and satisfies $\eta^{-1}\left(q\left(\left[v_{0},w_{0}\right]\right)\right)\subseteq B^{\prime}\subseteq\eta^{-1}\left(q\left(\left(v_{0}-r,w_{0}+r\right)\right)\right)$.
Since $q\restriction_{\Lambda}$ is a bijection, we get $B\cap B^{\prime}=\eta^{-1}\left(q\left(I_{0}^{\prime}\right)\right)$,
so $\eta^{-1}\left(q\left(I_{0}^{\prime}\right)\right)$ is definable
in $\cZ$.

Finally, since $I_{0}^{\prime}\subseteq K_{\xi\left(i\right)}$, the
length of $I_{0}^{\prime}$ is at most $\frac{1}{16m}\le\frac{1}{16}$.
\end{proof}
\begin{prop}
\label{definable_short_interval_starting_at_0}There exists $0<u<\frac{1}{4}$
such that $\eta^{-1}\left(q\left(\left(0,u\right)\right)\right)$
is definable in $\mathcal{Z}$.
\end{prop}

\begin{proof}
Let $I$ be as in \cref{definable_short_interval}, and denote
$J:=\eta^{-1}\left(q\left(I\right)\right)$. By replacing $J$ with
$J-a$ for some $a\in J$, we may assume that $0\in I$. By replacing
$J$ with $J\cup\left(-J\right)$, we may assume that $J=-J$, so
$I=-I$, but now the length of $I$ is at most $\frac{1}{8}$ instead
of $\frac{1}{16}$. So $I=\left(-u,u\right)$ for some $0<u\le\frac{1}{16}$.
Define a relation $R\left(x,y\right)$ by $R\left(a,b\right)\iff J\cap\left(J+a\right)\supseteq J\cap\left(J+b\right)$.
So $R$ is definable in $\mathcal{Z}$. 

For all $a,b\in\Z$ we have $b\in J+a\iff b-a\in J\iff\eta\left(b\right)-\eta\left(a\right)=\eta\left(b-a\right)\in q\left(I\right)\iff\eta\left(b\right)\in q\left(I\right)+\eta\left(a\right)=q\left(I\right)+q\left(\i{\eta\left(a\right)}\right)=q\left(I+\i{\eta\left(a\right)}\right)$,
so for all $a\in\Z$ we have $J+a=\eta^{-1}\left(q\left(I+\i{\eta\left(a\right)}\right)\right)$.
Also note that if $a\in J$ then $\i{\eta\left(a\right)}\in\i{q\left(I\right)}=I$
hence $I+\i{\eta\left(a\right)}\subseteq\left(-\frac{1}{8},\frac{1}{8}\right)$.
Therefore, since $q\restriction_{\left[-\frac{1}{2},\frac{1}{2}\right)}$
is a bijection, for all $a\in J$ we have $J\cap\left(J+a\right)=\eta^{-1}\left(q\left(I\cap\left(I+\i{\eta\left(a\right)}\right)\right)\right)$.
In general, for (possibly empty) open intervals $I_{1},I_{2}\subseteq\left(-\frac{1}{2},\frac{1}{2}\right)$
we have $I_{1}\subseteq I_{2}\iff\eta^{-1}\left(q\left(I_{1}\right)\right)\subseteq\eta^{-1}\left(q\left(I_{2}\right)\right)$:
Left to right is clear, so suppose $I_{1}\nsubseteq I_{2}$. Then
$I_{1}\backslash I_{2}$ contains a nonempty open interval. Since
$q\restriction_{\left(-\frac{1}{2},\frac{1}{2}\right)}$ is a homeomorphism,
$q\left(I_{1}\backslash I_{2}\right)$ contains a nonempty open set,
and since $\eta\left(\Z\right)$ is dense in $\R\big/\Z$, there exists
$d\in\Z$ such that $\eta\left(d\right)\in q\left(I_{1}\backslash I_{2}\right)=q\left(I_{1}\right)\backslash q\left(I_{2}\right)$.
So $d\in\eta^{-1}\left(q\left(I_{1}\right)\right)\backslash\eta^{-1}\left(q\left(I_{2}\right)\right)$.
In particular, for all $a,b\in J$ we have $R\left(a,b\right)\iff J\cap\left(J+a\right)\supseteq J\cap\left(J+b\right)\iff I\cap\left(I+\i{\eta\left(a\right)}\right)\supseteq I\cap\left(I+\i{\eta\left(b\right)}\right)$.

Fix $a\in J$ such that $\i{\eta\left(a\right)}>0$, and let $J^{\prime}:=\left\{ 0\neq b\in J\,:\,R\left(a,b\right)\text{ or }R\left(b,a\right)\right\} $.
So $J^{\prime}$ is definable in $\mathcal{Z}$. Note that $I\cap\left(I+\i{\eta\left(a\right)}\right)=\left(-u+\i{\eta\left(a\right)},u\right)$.
Let $b\in J$. If $\i{\eta\left(b\right)}>0$ then $I\cap\left(I+\i{\eta\left(b\right)}\right)=\left(-u+\i{\eta\left(b\right)},u\right)$,
so $R\left(a,b\right)\iff\i{\eta\left(a\right)}\le\i{\eta\left(b\right)}$
and $R\left(b,a\right)\iff\i{\eta\left(b\right)}\le\i{\eta\left(a\right)}$,
hence $b\in J^{\prime}$. If $\i{\eta\left(b\right)}<0$ then $I\cap\left(I+\i{\eta\left(b\right)}\right)=\left(-u,u+\i{\eta\left(b\right)}\right)$,
so we have both $\neg R\left(a,b\right)$ and $\neg R\left(b,a\right)$,
hence $b\notin J^{\prime}$. Therefore $J^{\prime}=\eta^{-1}\left(q\left(\left(0,u\right)\right)\right)$.
\end{proof}
From \cref{definable_short_interval_starting_at_0} and \cref{defining_the_cyclic_order_from_a_short_interval_starting_at_0}
we get that $C_{\alpha}$ is definable in $\cZ$, thus proving \cref{main_theorem_cyclic_order}.

\section{The converse to \texorpdfstring{\cref{main_theorem_cyclic_order}}{\ref{main_theorem_cyclic_order}}}
\begin{thm}
\label{if_cyclic_order_is_definable_then_G00_neq_G0}Let $\cZ$
be an expansion of $\left(\mathbb{Z},+,0,1,C_{\alpha}\right)$ for
some $\alpha\in\R\backslash\Q$, and let $G$ be a monster model.
Suppose that $G^{00}$ exists. Then $G^{00}\neq G^{0}$.
\end{thm}

\begin{proof}
Let $S:=\R\big/\Z$ and let $\cS:=\left(S,+,\cC\right)$, where $\cC$
is the positively oriented cyclic order on $S$. Let $\eta:\Z\rightarrow S$
be given by $\eta\left(n\right):=n\alpha+\Z$, and let $\cZ^{\prime}:=\left(\cZ,\cS,\eta\right)$.
Let $\fC^{\prime}$ be a monster model of $\cZ^{\prime}$, and let
$\fC$ be the reduct of $\fC^{\prime}$ to the language of $\cZ$.
So $\fC$ is a monster model of $\cZ$, and we may assume that $\fC=G$.
Denote by $\tilde{C}_{\alpha}$,$\tilde{S}$,$\tilde{\cC}$,$\tilde{\eta}$
the interpretations in $\fC^{\prime}$ of $C_{\alpha}$,$S$,$\cC$,$\eta$,
respectively. By definition, for all $a,b,c\in\Z$ we have $C_{\alpha}\left(a,b,c\right)\leftrightarrow\cC\left(\eta\left(a\right),\eta\left(b\right),\eta\left(c\right)\right)$,
so by elementarity, for all $a,b,c\in G$ we have $\tilde{C}_{\alpha}\left(a,b,c\right)\leftrightarrow\tilde{\cC}\left(\tilde{\eta}\left(a\right),\tilde{\eta}\left(b\right),\tilde{\eta}\left(c\right)\right)$.
Let $\text{st}:\tilde{S}\rightarrow S$ be the standard part map,
and let $h:=\text{st}\circ\tilde{\eta}$. So $h:G\rightarrow S$ is
a homomorphism and $h\restriction_{\Z}=\eta$. Since $\eta\left(\Z\right)$
is dense in $S$, by saturation we get that $h$ is surjective.
\begin{claim}
\label{continuity_of_the_homomorphism_to_the_circle}For every
closed subset $C\subseteq S$, $h^{-1}\left(C\right)$ is type-definable
over $\Z$ in $\cZ$.
\end{claim}

\begin{proof}[Proof of Claim]
Equivalently, for every open subset $U\subseteq S$, $h^{-1}\left(U\right)$
is $\bigvee$-definable over $\Z$ in $\cZ$. Let $q:\R\to S=\R\big/\Z$
be the quotient map. So for every $s\in\R$ we have that $q\restriction_{\left(s,s+1\right)}$
is a homeomorphism and that for every $s<u<v<s+1$, $q\left(\left(u,v\right)\right)$
is the subset of $S$ defined by $\cC\left(q\left(u\right),x,q\left(v\right)\right)$.
It is enough to show that for every $u,v\in\R$ such that $u<v<u+1$,
$h^{-1}\left(q\left(\left(u,v\right)\right)\right)$ is $\bigvee$-definable
over $\Z$ in $\cZ$. 

So let $u,v\in\R$ be such that $u<v<u+1$. Let $s\in\left(v,u+1\right)$
be such that $s\notin\alpha\Z+\Z$ (so $q\left(s\right)\notin\eta\left(\Z\right)$),
and let $\iota:=\left(q\restriction_{\left(s,s+1\right)}\right)^{-1}:\left(\R\big/\Z\right)\backslash\left\{ q\left(s\right)\right\} \to\left(s,s+1\right)$.
Since $\eta\left(\Z\right)$ is dense in $S$ and $\iota$ is a homeomorphism,
$\i{\eta\left(\Z\right)}$ is dense in $\left(s,s+1\right)$, so there
are sequences $\left(b_{i}\right)_{i<\omega}$,$\left(c_{i}\right)_{i<\omega}$
of elements in $\Z$ such that $u<\dots<\i{\eta\left(b_{2}\right)}<\i{\eta\left(b_{1}\right)}<\i{\eta\left(b_{0}\right)}<\i{\eta\left(c_{0}\right)}<\i{\eta\left(c_{1}\right)}<\i{\eta\left(c_{2}\right)}<\dots<v$,
$\inf_{i<\omega}\i{\eta\left(b_{i}\right)}=u$, and $\sup_{i<\omega}\i{\eta\left(c_{i}\right)}=v$.
For each $i<\omega$ let $D_{i}\subseteq G$ be the set defined by
$C_{\alpha}\left(b_{i},x,c_{i}\right)$. We show that $h^{-1}\left(q\left(\left(u,v\right)\right)\right)=\bigcup_{i<\omega}D_{i}$:

Let $a\in h^{-1}\left(q\left(\left(u,v\right)\right)\right)$. Then
$\i{h\left(a\right)}\in\left(u,v\right)=\bigcup_{i<\omega}\left(\i{\eta\left(b_{i}\right)},\i{\eta\left(c_{i}\right)}\right)$,
so for some $i<\omega$ we have $\i{h\left(a\right)}\in\left(\i{\eta\left(b_{i}\right)},\i{\eta\left(c_{i}\right)}\right)$.
Applying $q$ we get $\cC\left(\eta\left(b_{i}\right),h\left(a\right),\eta\left(c_{i}\right)\right)$,
i.e., $\cC\left(\eta\left(b_{i}\right),\text{st}\circ\tilde{\eta}\left(a\right),\eta\left(c_{i}\right)\right)$.
Therefore $\tilde{\cC}\left(\eta\left(b_{i}\right),\tilde{\eta}\left(a\right),\eta\left(c_{i}\right)\right)$,
and so $\tilde{C}_{\alpha}\left(b_{i},a,c_{i}\right)$, i.e., $a\in D_{i}$.

Suppose that $a\in D_{i}$ for some $i<\omega$. So $\tilde{C}_{\alpha}\left(b_{i},a,c_{i}\right)$,
therefore $\tilde{\cC}\left(\tilde{\eta}\left(b_{i}\right),\tilde{\eta}\left(a\right),\tilde{\eta}\left(c_{i}\right)\right)$,
and since $b_{i},c_{i}\in\Z$ we have $\tilde{\cC}\left(\eta\left(b_{i}\right),\tilde{\eta}\left(a\right),\eta\left(c_{i}\right)\right)$.
Since $s<\i{\eta\left(b_{i+1}\right)}<\i{\eta\left(b_{i}\right)}<\i{\eta\left(c_{i}\right)}<\i{\eta\left(c_{i+1}\right)}<s+1$
we get $\tilde{\cC}\left(\eta\left(b_{i+1}\right),\text{st}\circ\tilde{\eta}\left(a\right),\eta\left(c_{i+1}\right)\right)$,
i.e., $\cC\left(\eta\left(b_{i+1}\right),h\left(a\right),\eta\left(c_{i+1}\right)\right)$.
Therefore $h\left(a\right)\in q\left(\left(\i{\eta\left(b_{i+1}\right)},\i{\eta\left(c_{i+1}\right)}\right)\right)\subseteq q\left(\left(u,v\right)\right)$,
so $a\in h^{-1}\left(q\left(\left(u,v\right)\right)\right)$.
\end{proof}
In particular, $\cE:=h^{-1}\left(0\right)$ is type-definable over
$\Z$ in $\cZ$. By definition, $\cE$ has bounded index in $G$,
so $G^{00}\subseteq\cE$. By \cref{G0_in_any_expansion_of_Z} we
get $G^{00}\neq G^{0}$.
\end{proof}
\begin{rem}
In the context of \cref{if_cyclic_order_is_definable_then_G00_neq_G0},
it is worth noting that the converse to \cref{continuity_of_the_homomorphism_to_the_circle}
also holds, i.e., a subset $C\subseteq S$ is closed if and only if
$h^{-1}\left(C\right)$ is type-definable over $\Z$ in $\cZ$: $h$
induces a group isomorphism $\hat{h}:G\big/\cE\to S$, which is continuous
by \cref{continuity_of_the_homomorphism_to_the_circle}. Since
$G\big/\cE$ is compact, $\hat{h}$ is in fact a homeomorphism.
\end{rem}

\begin{rem}
In the context of \cref{if_cyclic_order_is_definable_then_G00_neq_G0},
it follows that for each $2\le m\in\Z$, $m\cE$ has bounded index
in $mG$ and hence in $G$. So $G^{00}\subseteq\bigcap_{m=1}^{\infty}m\cE$.
\end{rem}

\bibliographystyle{alpha}
\bibliography{on_dp_minimal_expansions_of_the_integers_2}

\newcommand{\etalchar}[1]{$^{#1}$}
\begin{thebibliography}{ADH{\etalchar{+}}16}

\bibitem[Ad19]{AD19}
Eran Alouf and Christian d'Elb\'{e}e.
\newblock A new dp-minimal expansion of the integers.
\newblock {\em J. Symb. Log.}, 84(2):632--663, 2019.

\bibitem[ADH{\etalchar{+}}16]{Aschenbrenner_et_al_I_2015}
Matthias Aschenbrenner, Alf Dolich, Deirdre Haskell, Dugald Macpherson, and Sergei Starchenko.
\newblock Vapnik-{C}hervonenkis density in some theories without the independence property, {I}.
\newblock {\em Trans. Amer. Math. Soc.}, 368(8):5889--5949, 2016.

\bibitem[Adl]{Adl}
Hans Adler.
\newblock Strong theories, burden, and weight.
\newblock Draft, available on the author's website.

\bibitem[Alo20]{Alo2020}
Eran Alouf.
\newblock On dp-minimal expansions of the integers.
\newblock {\em arXiv:2001.11480 [math.LO]}, January 2020.

\bibitem[Cla20]{Cla20}
Tim Clausen.
\newblock Dp-minimal profinite groups and valuations on the integers.
\newblock {\em arXiv:2008.08797 [math.LO]}, 2020.

\bibitem[Con18]{Conant2018_no_intermediate}
Gabriel Conant.
\newblock There are no intermediate structures between the group of integers and {P}resburger arithmetic.
\newblock {\em J. Symb. Log.}, 83(1):187--207, 2018.

\bibitem[CP18]{ConantPillay2018}
Gabriel Conant and Anand Pillay.
\newblock Stable groups and expansions of {$(\Bbb Z,+,0)$}.
\newblock {\em Fund. Math.}, 242(3):267--279, 2018.

\bibitem[DG17]{DolichGoodrick2017}
Alfred Dolich and John Goodrick.
\newblock Strong theories of ordered {A}belian groups.
\newblock {\em Fund. Math.}, 236(3):269--296, 2017.

\bibitem[She08]{She2008}
Saharon Shelah.
\newblock Minimal {B}ounded {I}ndex {S}ubgroup for {D}ependent {T}heories.
\newblock {\em Proceedings of the American Mathematical Society}, 136(3):1087--1091, 2008.

\bibitem[Sim15]{Simon_2015}
Pierre Simon.
\newblock {\em A guide to {NIP} theories}, volume~44 of {\em Lecture Notes in Logic}.
\newblock Association for Symbolic Logic, Chicago, IL; Cambridge Scientific Publishers, Cambridge, 2015.

\bibitem[SW19]{SW2019}
Pierre Simon and Erik Walsberg.
\newblock Dp and other minimalities.
\newblock {\em arXiv:1909.05399 [math.LO]}, September 2019.

\bibitem[Tao14]{Tao2014}
Terence Tao.
\newblock {\em Hilbert's {F}ifth {P}roblem and {R}elated {T}opics}.
\newblock American Mathematical Society, jul 2014.

\bibitem[TW23]{TW2023}
Chieu-Minh Tran and Erik Walsberg.
\newblock A {F}amily of dp-{M}inimal {E}xpansions of $(\mathbb{Z};+)$.
\newblock {\em Notre Dame Journal of Formal Logic}, 64(2), May 2023.

\bibitem[Wal20]{Wal2020}
Erik Walsberg.
\newblock Dp-minimal expansions of $(\mathbb{Z},+)$ via dense pairs via {M}ordell-{L}ang.
\newblock {\em arXiv:2004.06847 [math.LO]}, April 2020.

\end{thebibliography}

\end{document}